\documentclass[12pt]{amsart}
\usepackage{amssymb,latexsym,amsmath,amsthm,enumitem,mathrsfs,geometry}
\usepackage{fullpage}
\usepackage{fancyhdr}
\usepackage{xcolor}
\usepackage{bbm}
\usepackage{stmaryrd}
\usepackage{mathrsfs}
\usepackage[hypertexnames=false]{hyperref}
\usepackage{diagbox}
\usepackage{array}
\usepackage{tikz}
\usetikzlibrary{arrows}
\usetikzlibrary{positioning}
\usepackage{comment}
\usepackage{mathtools}
\usepackage{cleveref}

\newtheorem{theorem}{Theorem}[section]
\newtheorem*{theorem*}{Theorem}

\newtheorem{conjecture}{Conjecture}[section]
\newtheorem{heuristic}{Heuristic}[section]
\newtheorem{lemma}{Lemma}[section]
\newtheorem*{remark*}{Remark}

 \newtheorem{corollary}{Corollary}[section]
\newtheorem{proposition}[theorem]{Proposition}

\newtheorem{remark}{Remark}[section]

\newcommand{\eps}{\varepsilon}
\newcommand{\R}{\mathbb{R}}
\newcommand{\C}{\mathbb{C}}

\newcommand{\Z}{\mathbb{Z}}
\renewcommand{\r}{\operatorname{rad}}
\renewcommand{\arg}{\operatorname{arg}}
\newcommand{\dist}{\operatorname{dist}}
\newcommand{\Disp}{\operatorname{Disp}_{\ell^1}}
\renewcommand{\Re}{\operatorname{Re}}
\renewcommand{\Im}{\operatorname{Im}}

\newcommand{\Ann}{\operatorname{Ann}}

\begin{document}

\numberwithin{equation}{section}

\title{The maximal length of the Erd\H{o}s--Herzog--Piranian lemniscate in high degree}

\author[Tao]{Terence Tao}
\address{Department of Mathematics, UCLA, 405 Hilgard Ave, Los Angeles CA 90024}
\email{tao@math.ucla.edu}

\thanks{}

\date{\today}

\begin{abstract} Let $n \geq 1$, and let $p \colon \C \to \C$ be a monic polynomial of degree $n$.  It was conjectured by Erd\H{o}s, Herzog, and Piranian that the maximal length of lemniscate $\{z \in \C: |p(z)| = 1\}$ is attained by the polynomial $p(z) = z^n-1$.  In this paper, building upon a previous analysis of Fryntov and Nazarov, we establish this conjecture for all sufficiently large $n$.
\end{abstract}

\maketitle

\section{Introduction} \label{intro}

\subsection{The Erd\H{o}s--Herzog--Piranian conjecture}

Let $n \geq 1$, and let $p \colon \C \to \C$ be a monic polynomial of degree $n$.  We introduce the regions
$$ E_R = E_R(p) \coloneqq \{ z \in \C: |p(z)| < R \}$$
for any $R \geq 0$, with corresponding boundaries
$$ \partial E_R = \partial E_R(p) = \{ z \in \C: |p(z)| = R \}.$$
We refer to $\partial E_1(p) = \{ z \in \C: |p(z)| = 1 \}$ as the \emph{Erd\H{o}s--Herzog--Piranian lemniscate} (or \emph{lemniscate} for short) associated to $p$.  The problem of bounding the arclength $\ell(\partial E_1(p))$ of this lemniscate was raised (amongst many other questions) by Erd\H{o}s, Herzog, and Piranian in \cite{EHP}.  For instance, when $n=1$, $\partial E_1(p)$ is a unit circle and $\ell(\partial E_1(p)) = 2 \pi$.

As $r \to \infty$, the lemniscates $\partial E_1(z^2-r^2)$, which are Cassini ovals with foci $\pm r$ and eccentricity $e=1/r$, can be calculated to have length $O(1/r)$.  Thus the lemniscate length can be arbitrarily small without further hypotheses; however, in \cite[Problem 12]{EHP} it was conjectured that when one also required $E_1(p)$ to be connected, then one had the lower bound
$$ \ell(\partial E_1(p)) \geq \ell(\partial E_1(z^n)) = 2\pi.$$
This was quickly verified by Pommerenke \cite{pommerenke59}.

The question of upper bounds for $\ell(\partial E_1(p))$ was more difficult, with the following conjecture raised in \cite[Problem 12]{EHP} (see also \cite[p. 247]{er61}, \cite{er82e}, \cite{er90}, \cite{er97f}, \cite[Problem 114]{bloom}):

\begin{conjecture}[Erd\H{o}s--Herzog--Piranian conjecture]\label{ehp-conj}  If $n \geq 1$, and $p$ is monic of degree $n$, then
$$ \ell(\partial E_1(p)) \leq \ell(\partial E_1(p_0))$$
where $p_0(z) \coloneqq z^n-1$.
\end{conjecture}

\begin{remark}\label{normalize-rem} Note that the lemniscate length is unaffected if one replaces a polynomial $p(z)$ by a translate $p(z-z_0)$ or a rotation $e^{-in\theta} p(e^{i\theta} z)$ for $z_0 \in \C$ and $\theta \in \R$.  Thus for instance one could replace $p_0$ in the above conjecture by any translated, rotated version $(z-z_0)^n - e^{i\theta}$ of this polynomial (and it is natural to conjecture, especially in view of Theorem \ref{main-thm}(iv) below, that these are in fact the full set of maximizers of this length).  To eliminate these symmetries, we will find it convenient to work with polynomials $p$ that are \emph{normalized} in the sense that the $z^{n-1}$ coefficient vanishes and the constant coefficient is a non-positive real.  Then, $p_0$ is the only remaining conjectural candidate for a maximizer.
\end{remark}

From the above discussion, Conjecture \ref{ehp-conj} holds for $n=1$. To avoid minor degeneracies, we will now assume $n \geq 2$ henceforth.  Numerical experimentation by the author using the tool \emph{AlphaEvolve} suggested that this conjecture was true for all $n$, as the tool quickly settled on $p_0$ (or a rotation and translation thereof) as a potential maximizer of $\ell(\partial E_1(p))$ for various choices of degree $n$, though this fell well short of a rigorous proof.

A routine calculation (see Lemma \ref{out}) shows that
\begin{equation}\label{lbo}
 \ell(\partial E_1(p_0)) = 2^{1/n} B\left( \frac{1}{2}, \frac{1}{2n} \right)
\end{equation}
where $B$ is the beta function.  In particular, from the Stirling approximation one can calculate that
\begin{equation}\label{ep0}
     \ell(\partial E_1(p_0)) = 2n + 4 \log 2 + O\left(\frac{1}{n}\right) = 2n + O(1).
\end{equation}
as $n \to \infty$.

\begin{figure}[t]
\centering
\includegraphics[width=2.75in]{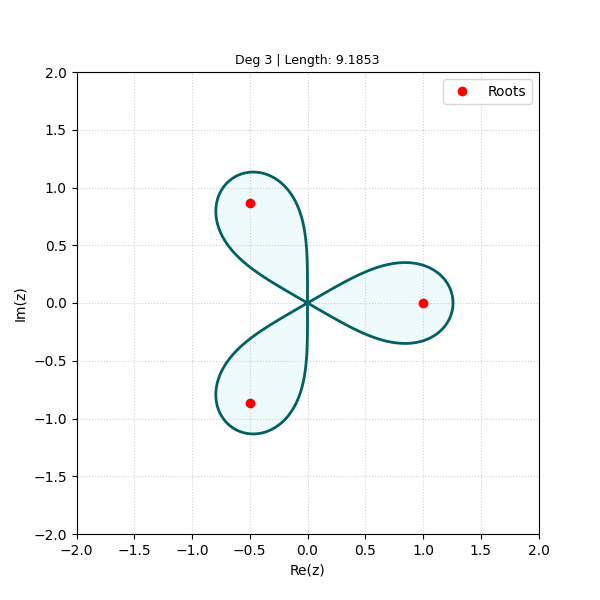}
\includegraphics[width=2.75in]{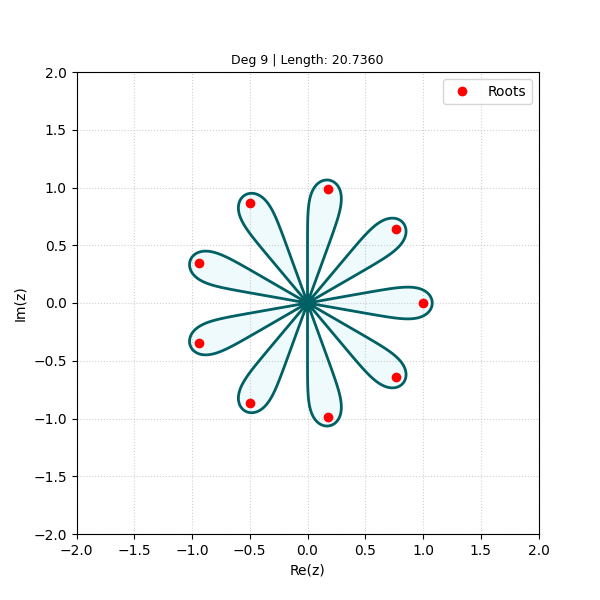}
\caption{The lemniscate $\partial E_1(p_0)$ for $n=3,9$.  The approximation $2n + 4 \log 2$ to the lemniscate length is $\approx 8.773$ for $n=3$ and $\approx 20.773$ for $n=9$. The initial code for this and subsequent plots were provided by AlphaEvolve and Gemini.}
\label{fig:p0}
\end{figure}

\begin{remark}
It is also instructive to try to establish the bound \eqref{ep0} geometrically.  The region $E_1(p_0)$ decomposes into $n$ thin ``petals'' emenating from the origin, while the lemniscate $\partial E_1(p_0)$ consists of $2n$ ``spokes'' emenating from the origin that are approximately unit line sgments, together with $n$ ``tips'' resembling semicircles of radius $1/n$, for a total length of approximately $2n+O(1)$; see Figure \ref{fig:p0}.  The spokes arise from the observation that in the bulk region $D(0,1-C/n)$ for a large constant $C$, the condition $|p_0(z)|=1$ is approximately equivalent to the condition that $z^n$ lies on the imaginary axis.  If instead one takes an ansatz $z = e^{2\pi i j/n} (1 + w/n)$ to study the $j^{\mathrm{th}}$ tip, the condition $|p_0(z)|=1$ is approximately equivalent to the condition that $|e^w-1|=1$; see Figure \ref{fig:tip}.  See Lemma \ref{out} for a more rigorous justification of these heuristics.
\end{remark}

\begin{figure}[t]
\centering
\includegraphics[width=5in]{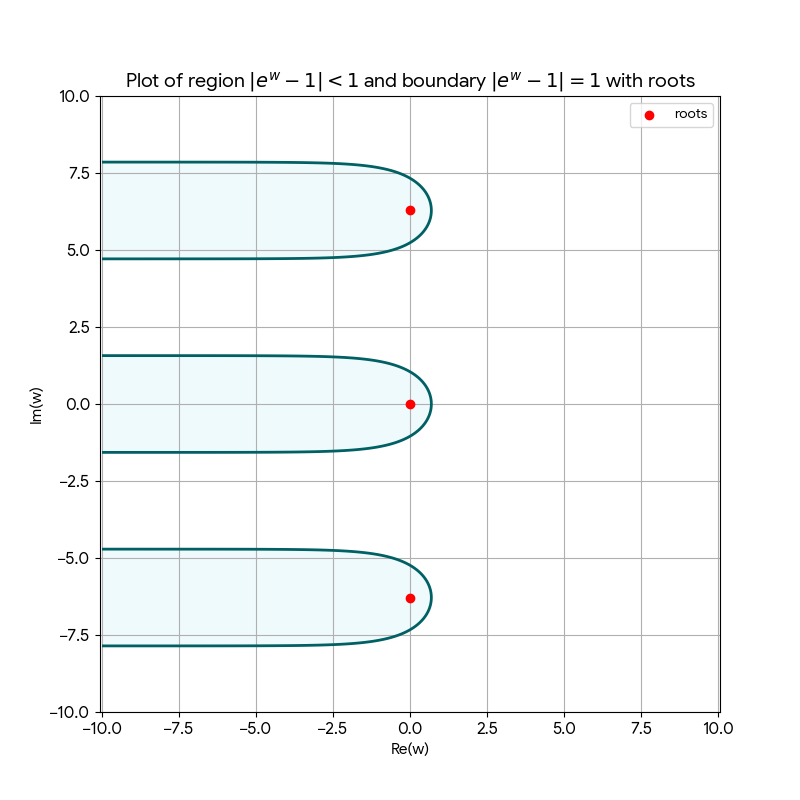}
\caption{The curve $\{ w: |e^w - 1| = 1\}$, together with the roots $w = 2k \pi i$ of $e^w-1=0$.  This is the asymptotic profile of the lemniscate $E_1(p_0)$ as $n \to \infty$, after rescaling around an $n^{\mathrm{th}}$ root of unity.}
\label{fig:tip}
\end{figure}

There are several results in the literature obtaining upper bounds on $\ell(\partial E_1(p))$ in the direction of establishing Conjecture \ref{ehp-conj}: see Table \ref{results-table}.  In particular, prior to this work, this conjecture was settled for $n=2$ \cite{eremenko}, and for $p$ sufficiently close to $p_0$ \cite[\S 6]{fryntov}, and up to an error of $O(n^{7/8})$ in the general case \cite[\S 8]{fryntov}.  In \cite[\S 8]{fryntov} the authors write ``We have no doubt that the power $7/8$ can be substantially improved though
to bring it below $1/2$ seems quite a challenging problem.''

\begin{table}[ht]
\begin{tabular}{|c|c|c|}
\hline
Upper bound & Hypotheses & Reference \\
\hline
$74 n^2$ && Pommerenke \cite{pommerenke61}\\
$8e\pi n$ && Borwein \cite{borwein} \\
$4\pi n$ && Dol\v{z}enko \cite{dolzenko} \\
$4\pi n - 2\pi$ && Fryntov--Nazarov \cite[\S 7]{fryntov} \\
$9.173 n$ && Eremenko--Hayman \cite{eremenko} \\
$2\pi n + 2\pi \sqrt{n} - 2\pi$ && Fryntov--Nazarov \cite[\S 8]{fryntov} \\
$2\pi n$ && Danchenko \cite{danchenko} \\
$\pi n + O(\sqrt{n \log n})$ & $n \geq 2$ & Kosukhin \cite[Theorem 3]{kosukhin} \\
\hline
$2n + O(n^{7/8})$ && Fryntov--Nazarov \cite[\S 9]{fryntov}\\
$\ell(\partial E_1(p_0))$ &$n=1$ & trivial\\
$\ell(\partial E_1(p_0))$ &$n=2$ & Eremenko--Hayman \cite{eremenko} \\
$\ell(\partial E_1(p_0))$ &$p$ close to $p_0$ & Fryntov--Nazarov \cite[\S 6]{fryntov}\\
\hline
$2n + O(\sqrt{n})$ && Theorem \ref{main-thm}(i) \\
$2n + O(1)$ && Theorem \ref{main-thm}(ii) \\
$2n + 4 \log 2 + o(1)$ && Theorem \ref{main-thm}(iii) \\
$\ell(\partial E_1(p_0))$ & $n$ sufficiently large & Theorem \ref{main-thm}(iv) \\
\hline
\end{tabular}
\caption{A summary of upper bounds on the lemniscate length $\ell(\partial E_1(p))$ for monic polynomials $p$ of degree $n \geq 2$, including the new results in this paper.}\label{results-table}
\end{table}

\subsection{Main results}

Our main results sharpen the upper bounds, and in fact establish the Erd\H{o}s--Herzog--Piranian conjecture for sufficiently large $n$:

\begin{theorem}[Main theorem]\label{main-thm}  Let $p$ be a monic polynomial of degree $n$.
\begin{itemize}
    \item[(i)] One has $\ell(\partial E_1(p)) \leq 2n + O(\sqrt{n})$.
    \item[(ii)] One has $\ell(\partial E_1(p)) \leq 2n + O(1)$.
    \item[(iii)] One has $\ell(\partial E_1(p)) \leq 2n + 4 \log 2 + o(1)$ as $n \to \infty$.
    \item[(iv)] If $n$ is sufficiently large, then $\ell(\partial E_1(p)) \leq \ell(\partial E_1(p_0))$.  Furthermore, equality is attained if and only if $p$ is equal to $p_0$ up to rotation and translation (i.e., $p(z) = (z-z_0)^n - e^{i\theta}$ for some $z_0 \in \C$ and $\theta \in \R$).
\end{itemize}
\end{theorem}

Of course, each part of this theorem implies the previous ones, so one could retain only the strongest conclusion (iv) and discard the other three components (i), (ii), (iii) of the theorem; but our proof will be structured by establishing (i), (ii), (iii), and (iv) in turn, with the proof of each part relying on the preceding ones.  Part (iv) has some resemblance to a previous result \cite{tao-sendov} of the author establishing Sendov's conjecture for sufficiently high degree polynomials, but the proof techniques here are rather different from those in that paper.

\begin{remark}  All implied constants in our arguments are effectively computable.  Thus, in view of Theorem \ref{main-thm}(iv), the full verification of Conjecture \ref{ehp-conj} now reduces to checking the conjecture for an explicitly bounded number of $n$, although we have made no attempt to optimize this bound.  This \emph{almost} allows us to declare this conjecture to be decidable; however, there is still one (very unlikely) scenario which could obstruct this conjecture from being decided in finite time, in which there is some bounded degree $n$ and some competitor normalized extremizer $p_1 \neq p_0$ for which the lemniscate lengths $\ell(\partial E_1(p_0))$, $\ell(\partial E_1(p_1))$ happen to be identical.  Due to the (likely) transcendental nature of these lengths, it is not immediately obvious that one could decide between the rival scenarios $\ell(\partial E_1(p_0)) < \ell(\partial E_1(p_1))$, $\ell(\partial E_1(p_0)) = \ell(\partial E_1(p_1))$, or $\ell(\partial E_1(p_0)) > \ell(\partial E_1(p_1))$ by a finite computation.
\end{remark}

\subsection{Methods of proof}

We now discuss the methods of proof of Theorem \ref{main-thm}.  It will be convenient to use the following result of Eremenko and Hayman \cite{eremenko}:

\begin{proposition}\label{erem}  Let $n \geq 1$.  Then there exists a monic polynomial $p$ of degree $n$ which maximizes $\ell(\partial E_1(p))$ among all such polynomials.  Furthermore, the lemniscate $\partial E_1(p)$ is connected and contains all the critical points of $p$.  Finally, we can assume $p$ to be normalized in the sense of Remark \ref{normalize-rem}.
\end{proposition}

\begin{proof} The existence of a maximizing $p$ whose lemniscate is connected and contains all of its critical points follow from\footnote{We remark that the proofs of these lemmas are non-elementary, relying on the theory of quasiconformal mapping as well as the Riemann--Hurwitz formula.  It seems of interest to obtain more elementary proofs of these results.  The fact that the lemniscate contains its critical points is not essential for our arguments, as we can use the Gauss--Lucas theorem as a substitute; but the connectedness will be used in a few places, namely in Section \ref{area-sec} and in the proof of Lemma \ref{initial-apps}.} \cite[Lemmas 5,6]{eremenko}.  One can then replace $p(z)$ by a translation $p(z-z_0)$ so that the $z^{n-1}$ coefficient vanishes, and then replace $p(z)$ by a rotation $e^{-in\theta} p(e^{i\theta} z)$ so that $p(0)$ is a non-positive real.
\end{proof}

Let us refer to a polynomial $p$ obeying all the properties of Proposition \ref{erem} as a \emph{normalized maximizer}.  For instance, when $n=2$ the only possible normalized maximizer is $p_0(z) = z^2-1$, which explains why the $n=2$ case of Conjecture \ref{ehp-conj} was already established in \cite{eremenko}.  However, these properties do not uniquely identify the normalized maximizer for $n \geq 3$; see Figure \ref{fig:compete}.

\begin{figure}[t]
\centering
\includegraphics[width=5in]{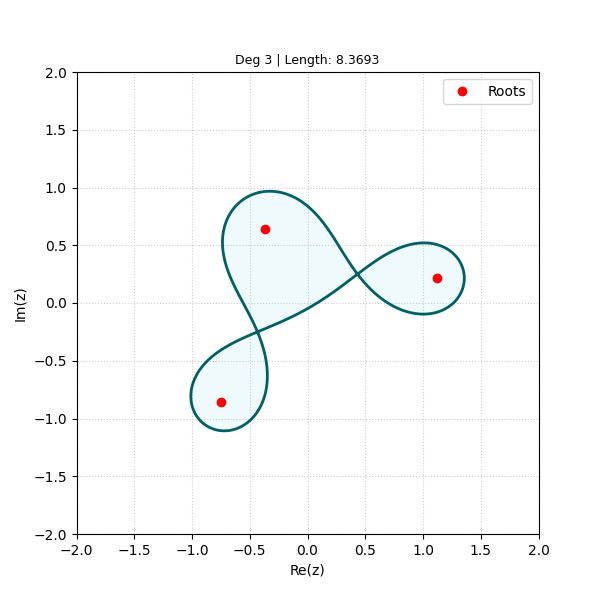}
\caption{The lemniscate $\partial E_1(z^3 - 3e^{i\pi/3} a^2 z - \sqrt{1-4a^6})$ for $a=1/2$.  This lemniscate is connected and contains the two critical points $\pm e^{i\pi/6} a$ (as self-crossings), and is normalized, but is not a maximizer, since its length is less than that of $\partial E_1(z^3-1)$.}
\label{fig:compete}
\end{figure}

To prove Theorem \ref{main-thm}, it clearly suffices to do so for normalized maximizers; in particular, (iv) is equivalent to the assertion that $p_0$ is the only normalized maximizer when $n$ is sufficiently large.

An important role in the arguments\footnote{In contrast, the $n$ zeroes of $p$ shall play almost no role whatsoever.} will be played by the $n-1$ critical points $\zeta$ of $p$ (the roots of $p'$, counting multiplicity), as well as the constant term $p(0)$ of the polynomial.  Note that from the fundamental theorem of algebra we have
\begin{equation}\label{pp-factor}
    p'(z) = n \prod_\zeta (z-\zeta)
\end{equation}
and hence by the fundamental theorem of calculus, $p$ is completely determined by the critical points $\zeta$ and the constant term $p(0)$.
For instance, we have $p=p_0$ if and only if all the critical points $\zeta$ vanish and $p(0)=-1$.  To measure the deviation of $p$ from $p_0$, we introduce the ($\ell^1$) \emph{dispersion}
\begin{equation}\label{semi-norm}
    \|p\|_1 \coloneqq \sum_\zeta |\zeta|,
\end{equation}
the \emph{origin repulsion}
\begin{equation}\label{origin-repulse}
    \|p\|_0 \coloneqq n |1 + p(0)|^{1/n},
\end{equation}
and the \emph{total size}
\begin{equation}\label{size-def}
    \|p\| \coloneqq \|p\|_1 + \|p\|_0
\end{equation}
Clearly we have
\begin{equation}\label{sdef}
    0 \leq \|p\|_1, \|p\|_0 \leq \|p\|
\end{equation}
and $\|p\|$ vanishes if and only if $p=p_0$.  Thus one can view the quantities $\|p\|_1, \|p\|_0, \|p\|$ as measuring the extent to which $p$ deviates from $p_0$.  For future reference, we make the simple but useful observation (from Markov's inequality) that $\|p\|_1$ (and hence $\|p\|$) limits the number of critical points $\zeta$ that lie outside of a disk, in that
\begin{equation}\label{markov}
    \# \{ \zeta: \zeta \notin D(0,r) \} \leq \frac{\|p\|_1}{r} \leq \frac{\|p\|}{r}
\end{equation}
for any $r>0$ (where $\zeta$ ranges over critical points counted with multiplicity).  In particular, as there are only $n-1$ critical points in all, one has (assuming $n>2$)
\begin{equation}\label{markov-many}
    \# \{ \zeta: \zeta \in D(0,2\|p\|_1/n) \} \geq n-1-\frac{n}{2} \gg n.
\end{equation}

Also, as $p(0)$ is normalized to be a non-positive real, we have
$$ \dist(p(0), \partial D(0,1)) = |1+p(0)| = n^{-n} \|p\|_0^n$$
and hence
\begin{equation}\label{npz}
    |p(z) - p(0)| \geq n^{-n} \|p\|_0^n
\end{equation}
for any $z \in \partial E_1(p)$. This already hints at why we view $\|p\|_0$ as a measure of origin repulsion; see Lemma \ref{origin-repulsion} for a further development of this inequality.

\begin{remark}\label{l2}  Another natural measure of deviation from $p_0$ is the \emph{$\ell^2$ dispersion} (or \emph{unnormalized variance})
\begin{equation}\label{l2-def}
 \|p\|_2 \coloneqq \sum_\zeta |\zeta|^2.
\end{equation}
This quantity will occasionally arise in our arguments, but in view of the trivial bound
\begin{equation}\label{l1-l2}
    \|p\|_2 \leq \|p\|_1^2,
\end{equation}
control on $\|p\|_2$ is not as desirable as control on $\|p\|_1$.  For instance, by \eqref{markov-many}, a bound of the form $\|p\|_1 = O(1)$ will localize most of the critical points to lie within $O(1/n)$ of the origin, whereas the weaker bound $\|p\|_2 = O(1)$ merely localizes most of the critical points to within $O(1/\sqrt{n})$ of the origin.  As such, we will avoid relying on the $\ell^2$ dispersion in our arguments.
\end{remark}

A guiding heuristic for our analysis is as follows:

\begin{heuristic}\label{main-heuristic}  For normalized polynomials $p$, the lemniscate length $\ell(\partial E_1(p))$ ``behaves like'' $\ell(\partial E_1(p_0)) - c \|p\|$ for some constant $c>0$.  In particular:
\begin{itemize}
    \item[(i)] Any increase in the dispersion $\|p\|_1$ is expected to induce a proportional decrease in the lemniscate length.
    \item[(ii)] Any increase in the origin repulsion $\|p\|_0$ is expected to induce a proportional decrease in the lemniscate length.
    \item[(iii)]  As $\ell(\partial E_1(p))$ approaches $\ell(\partial E_1(p_0))$, both the dispersion $\|p\|_1$ and the origin repulsion $\|p\|_0$ are expected to approach zero, and the lemniscate $\partial E_1(p)$ should increasingly resemble $\partial E_1(p_0)$ in shape.
\end{itemize}
\end{heuristic}

We now pause to illustrate this heuristic with some examples.  Let us first consider the normalized polynomial
\begin{equation}\label{example-1}
 p(z) = z^n - \frac{n}{n-2} a^2 z^{n-2} - 1
\end{equation}
for some $n>2$ and $0 < a < 1$.  Here, there is no origin repulsion: $\|p\|_0=0$.  However, the dispersion $\|p\|_1$ is equal to $2a$, since two of the critical points have moved from $0$ to $\pm a$ when compared against $p_0$.  When $n$ is odd, what happens to the $n$ ``petals'' of the lemniscate $\partial E_1(p_0)$ when one replaces it with $\partial E_1(p)$ is that one of the petals retreats\footnote{Because of this, the lemniscate is no longer connected, and no longer contains all of its critical points.  However, this does not contradict Proposition \ref{erem}, since $p$ is no longer a maximizer.  Similarly for the example \eqref{example-2} below.} a distance $\approx a$ from the origin, and two other opposing petals merge for a similar distance, shortening the total length by $\approx 4a = 2 \|p\|_1$; see Figure \ref{fig:odd}.  When $n$ is even, one instead has two petals retreat by distance $\approx a$, again shortening the lemniscate length by $\approx 4a = 2 \|p\|_1$; see Figure \ref{fig:even}. From the perspective of the argument principle (or Rouche's theorem), this reflects the fact that the polynomial $p(z)$ is well approximated by $p_0(z)$ for $|z| \gg a$, but is instead approximated by $\frac{n}{n-2} a^2 z^{n-2} - 1$ for $|z| \ll a$, causing the $2n$ ``spokes'' in the lemniscate to drop down to $2(n-2)$ in this region.

\begin{figure}[t]
\centering
\includegraphics[width=5in]{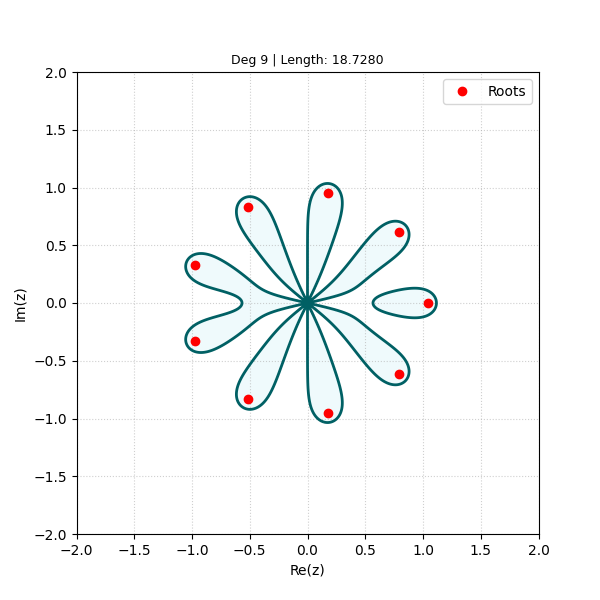}
\caption{The lemniscate $\partial E_1(z^n - \frac{n}{n-2} a^2 z^{n-2} - 1)$ with $n=9$ and $a=1/2$, which has shortened in length from $\partial E_1(p_0)$ by approximately $4a = 2\|p\|_1 = 2$.}
\label{fig:odd}
\end{figure}

\begin{figure}[t]
\centering
\includegraphics[width=5in]{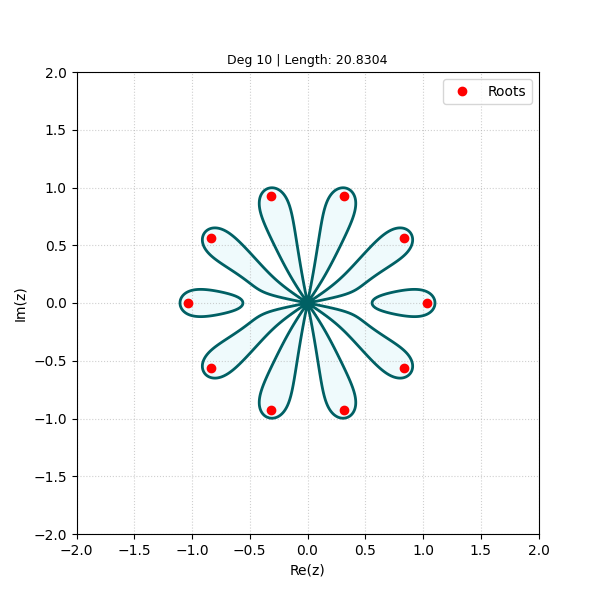}
\caption{The lemniscate $\partial E_1(z^n - \frac{n}{n-2} a^2 z^{n-2} - 1)$ with $n=10$ and $a=1/2$, which experiences a similar length shortening to the example in Figure \ref{fig:odd}, albeit with a different topology caused by the change in the parity of $n$.}
\label{fig:even}
\end{figure}

The second key example is the normalized polynomial
\begin{equation}\label{example-2}
 p(z) = z^n - 1 - (a/n)^n
\end{equation}
for some $0 < a \ll 1$.  Here there is no dispersion as all the critical points remain at the origin: $\|p\|_1=0$. On the other hand, the origin repulsion $\|p\|_0$ is equal to $a$.  Effectively, the $(a/n)^n$ term repels the $2n$ spokes of the lemniscate $\partial E_1(p_0)$ outwards by a distance $\approx a/n$, shortening the lemniscate length by an amount slightly less\footnote{There is a lower order correction term comparable to the circumference $2\pi a/n$ of $\partial D(0,a/n)$.} than $2a = 2\|p\|_0$; see Figure \ref{fig:repel}.

\begin{figure}[t]
\centering
\includegraphics[width=5in]{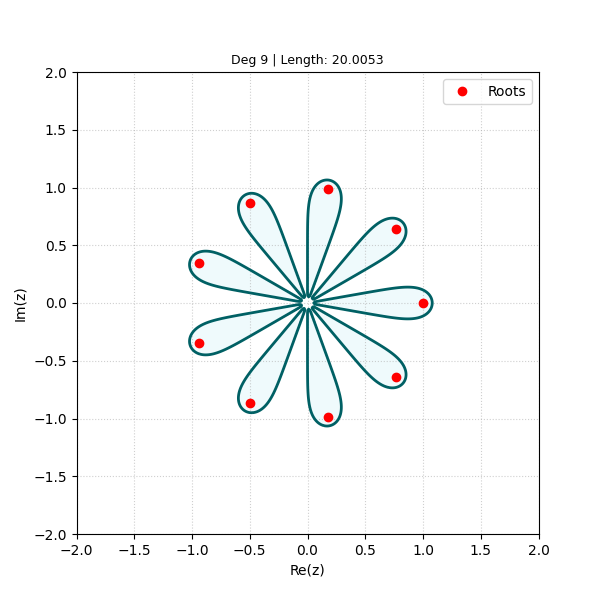}
\caption{The lemniscate $\partial E_1(z^n - 1 - (a/n)^n)$ with $n=9$ and $a=1/2$, which has shortened in length from $\partial E_1(p_0)$ by slightly less than $2a = 2\|p\|_0 = 1$. due to the $n$ petals of the lemniscate being repelled from the origin by a distance $\approx a/n$.}
\label{fig:repel}
\end{figure}

It remains to make these heuristics rigorous.  The first step in this direction is to follow Fryntov and Nazarov \cite{fryntov} and use Stokes' theorem to represent the lemniscate length $\ell(\partial E_1(p))$ as an area integral.  There are multiple useful representations of this type.  In \cite{fryntov}, the representations
\begin{equation}\label{length-1}
    \ell(\partial E_1(p)) = - \int_{E_1(p)} \frac{|\varphi|}{\varphi} \frac{\varphi'}{\varphi}\ dA
\end{equation}
and
\begin{equation}\label{length-2}
    \ell(\partial E_1(p)) = 2 \int_{E_1(p)} |p'|\ dA - \int_{E_1(p)} \frac{|p\varphi|}{\varphi} \psi\ dA
\end{equation}
were utilized, where $dA$ denotes area measure,
\begin{equation}\label{phi-def}
    \varphi(z) \coloneqq \frac{p'(z)}{p(z)}
\end{equation}
is the logarithmic derivative of $p$, and
\begin{equation}\label{psi-def}
    \psi(z) \coloneqq \frac{p''(z)}{p'(z)}
\end{equation}
is the logarithmic derivative of $p'$.  The function $\psi$ will play a central role in our arguments. Note from \eqref{pp-factor} that
\begin{equation}\label{psi-expand}
    \psi(z) = \sum_\zeta \frac{1}{z-\zeta},
\end{equation}
which by the triangle inequality yields the trivial upper bound
\begin{equation}\label{triangle}
    |\psi(z)| \leq \sum_\zeta \frac{1}{|z-\zeta|}.
\end{equation}
As we shall see, an analysis of the defect in \eqref{triangle} will be crucial in obtaining the later components of Theorem \ref{main-thm}.

The representation \eqref{length-1}, when combined with \eqref{triangle} and the classical P\'olya inequality \cite{polya}
\begin{equation}\label{polya}
    |E_r(p)| \leq |E_r(z^n)| = \pi r^{2/n},
\end{equation}
(with $|E|$ denoting the area of $E$) was shown in \cite[\S 7]{fryntov} to easily lead to the bound $\ell(\partial E_1(p)) \leq 4\pi n - 2\pi$; in \cite[\S 8]{fryntov} a similar argument used \eqref{length-2} to obtain the improvement $\ell(\partial E_1(p)) \leq 2\pi n + 2\pi \sqrt{n} - 2\pi$.  Finally, in \cite[\S 9]{fryntov} the identity \eqref{length-2}, combined with additional decompositions and an integration by parts, was used to eventually obtain the asymptotically superior bound $\ell(\partial E_1(p)) \leq 2n + O(n^{7/8})$.

In Section \ref{stokes-sec} below, we modify the above identities by inserting various cutoffs, and also implement the integration by parts argument from \cite[\S 9]{fryntov} to ``homogenize'' the integrand.  This lets us obtain the following flexible and useful variant of \eqref{length-1}, \eqref{length-2}.  We define an absolutely continuous measure $\Psi$ on the complex plane by the formula
\begin{equation}\label{Psi-def}
    \Psi(E) \coloneqq \frac{1}{\pi} \int_E |\psi|\ dA.
\end{equation}
The key result below asserts, roughly speaking, that this absolutely continuous measure $\Psi$ is a good (upper bound) proxy\footnote{This phenomenon is faintly reminscent of the landscape function being a good proxy for the distribution of eigenfunctions of a Schr\"odinger operator \cite{mayboroda}.  We do not know if there is any deeper insight to be drawn from this resemblance.} for the arclength measure on the lemniscale $\partial E_1(p)$; this is illustrated visually in Figure \ref{fig:heat}.

\begin{figure}[t]
\centering
\includegraphics[width=5in]{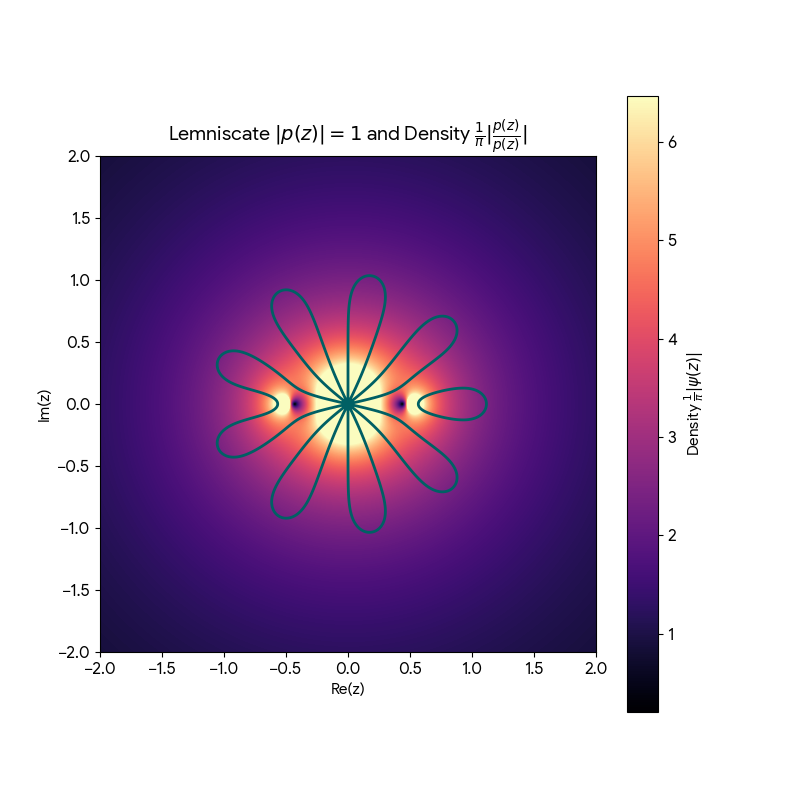}
\caption{The lemniscate from Figure \ref{fig:odd}, superimposed upon a heat map of the density $\frac{1}{\pi} |\psi|$ of $\Psi$.  The dispersion of the zeroes creates ``holes'' in the heat map between the critical points $0, \pm 1/2$ due to defects in the triangle inequality \eqref{triangle}.  These holes correspond to the portions of the lemniscate that have retreated from the origin.}
\label{fig:heat}
\end{figure}

\begin{theorem}[Application of Stokes' theorem]\label{stokes}
    Let $p$ be a polynomial, $\Omega \subset \C$ a semi-algebraic\footnote{That is to say, $\Omega$ is defined by a finite number of polynomial inequalities involving the real and imaginary part of the variable $z$.  In particular, the boundary of $\Omega$ will be piecewise smooth by the semi-algebraic stratification theorem \cite{bochnak}.} open set, and $\lambda \colon \C \to \R^+$ be an arbitrary smooth everywhere positive function. Then we have the estimate
\begin{equation}\label{stokes-eq}
    \ell(\partial E_1(p) \cap \overline{\Omega}) \leq \Psi(E_2(p) \cap \Omega) + O( X_1 + X_2 + X_3 + X_4 + X_5),
\end{equation}
where $\psi$ was given by \eqref{psi-def}, and the error terms $X_1,X_2,X_3,X_4,X_5$ are given by
\begin{align}
X_1 &\coloneqq \int_{E_2(p) \cap \Omega} |p'|\ dA \label{X1-def} \\
X_2 &\coloneqq \int_{E_2(p) \cap \Omega, |\psi| \leq \lambda} |\psi|\ dA \label{X2-def}\\
X_3 &\coloneqq \int_{E_2(p) \cap \Omega, |\psi| \geq \lambda/2} \frac{|\psi'|}{|\psi|}\ dA \label{X3-def}\\
X_4 &\coloneqq \int_{E_2(p) \cap \Omega} \frac{|\lambda'|}{|\lambda|}\ dA \label{X4-def}\\
X_5 &\coloneqq \ell(\partial \Omega). \label{X5-def}
\end{align}
\end{theorem}

A key feature here is the additional factor of $\frac{1}{\pi}$ in \eqref{Psi-def}, which is not present in \eqref{length-1} or \eqref{length-2}, and is the key to obtaining estimates that are close to $2n$ rather than to $2\pi n$ or $4\pi n$.
Roughly speaking, the restriction to $E_1(p)$ in \eqref{length-1}, \eqref{length-2}, together with the phase $\frac{|\varphi|}{\varphi}$, introduces a weight that behaves like $\cos_-(x) \coloneqq \max(-\cos x, 0)$ where $x$ is related to the argument of $\varphi$.  As observed in \cite[(16)]{fryntov}, the average of this weight is $\frac{1}{\pi}$, which heuristically explains \eqref{stokes-eq}; see Figure \ref{fig:cos}. Integration by parts is used to make this heuristic rigorous.

\begin{figure}[t]
\centering
\includegraphics[width=5in]{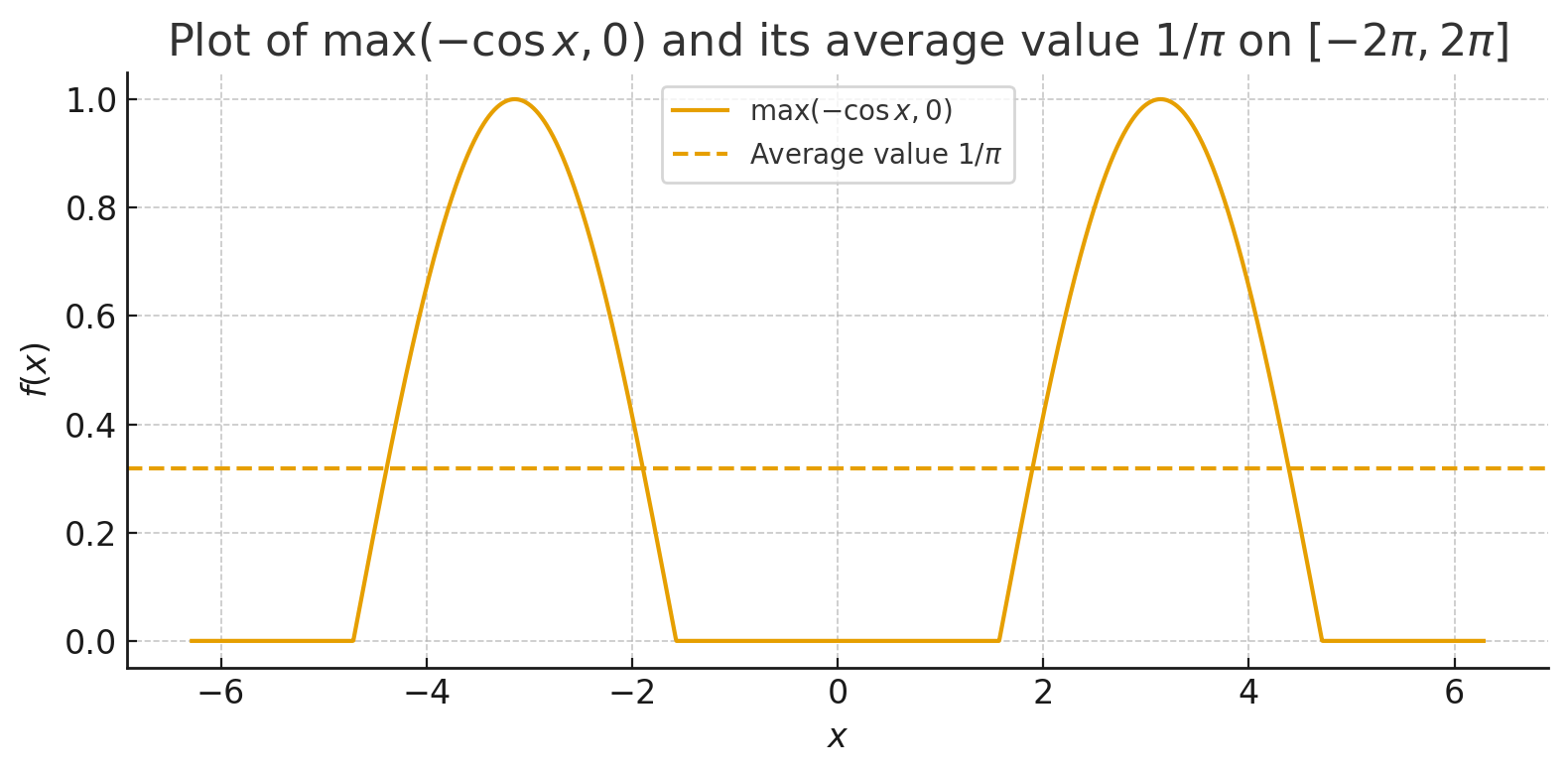}
\caption{An illustration (generated by ChatGPT Pro) of the easily verified fact that the average value of the $2\pi$-periodic function $\cos_-(x) = \max(-\cos x, 0)$ is $\frac{1}{\pi}$.  The regions where $\cos_-$ is positive correspond, roughly speaking, to how $E_1$ is situated within $E_2$ in, say, Figure \ref{fig:lemniscates}.}
\label{fig:cos}
\end{figure}

\begin{remark} Another way to heuristically justify the link between $\ell(\partial E_1(p) \cap \Omega)$ and $\Psi(E_2(p) \cap \Omega)$ in Theorem \ref{stokes} as follows.  Suppose we work in a region $\Omega$ where the functions $\psi$, $p$ are roughly constant: $\psi(z) \approx \psi_0$, $p(z) \approx c_0$.  For simplicity let us normalize $\psi_0$ to be real, and $c_0$ to be negative real.  In order to have a non-trivial lemniscate in this region, $c_0$ should be close to $-1$, in which case we expect
    \begin{equation}\label{psio}
        \Psi(E_2(p) \cap \Omega) \approx \frac{\psi_0}{\pi} |\Omega|.
    \end{equation}
  Because the unit circle $\partial D(0,1)$ is tangent to the line $\{ w: \Re w = -1\}$ at $-1$, the lemniscate condition $|p(z)|=1$ is then heuristically approximated by the condition that $\Re (p(z) + 1) = 0$.  On the other hand, the hypothesis $\psi(z) \approx \psi_0$ and \eqref{psi-def} suggests that $p'(z) \approx A e^{\psi_0 z}$ for some amplitude $z$, which heuristically integrates to $p(z) \approx c_0 + \frac{A}{\psi_0} e^{\psi_0 z}$.  Writing $\frac{A}{\psi_0}$ in polar coordinates as $Re^{i\theta}$ and $z$ in Cartesian coordinates as $x+iy$, the condition $\Re(p(z)+1)=0$ can then be rearranged after some algebra as
$$ \cos (\psi_0 y + \theta) \approx -\frac{1+c_0}{Re^{\psi_0 x}}.$$
If the right-hand side is much larger than $1$ in magnitude, we thus expect the lemniscate to be empty in this region; but if instead the right-hand side is much less than $1$ in magnitude, we expect the lemniscate to behave like a periodic sequence of horizontal lines of spacing $\frac{\pi}{\psi_0}$.  Thus, the total length of the lemniscate in this region $\Omega$ would be expected to not exceed $\frac{\psi_0}{\pi} |\Omega|$ up to lower order errors; comparing this with \eqref{psio} we arrive at a heuristic justification of \eqref{stokes-eq}.
\end{remark}

In practice, the various error terms $X_1,\dots,X_5$ can by controlled (for suitable choices of $\lambda$ and $\Omega$) by a variety of tools, including classical area inequalities (see Section \ref{area-sec}), rearrangement inequalities (see Section \ref{riesz-sec}), and partial fraction decompositions (see Section \ref{disp-sec}). It is an easy matter to obtain Theorem \ref{main-thm}(i) from Theorem \ref{stokes} and the aforementioned tools to control error terms; this can be viewed as an optimized version of the arguments in \cite{fryntov} (who, as mentioned previously, already identified the $O(\sqrt{n})$, as a natural barrier to their method). We carry out these arguments in Section \ref{main-i-sec}.

To proceed further, we analyze the defect in various inequalities used in the proof of Theorem \ref{main-thm}(i).  It turns out that the most fruitful component inequality to analyze in this fashion is the triangle inequality \eqref{triangle}.  One can use a defect version of this inequality (see Section \ref{defect-sec}) to control the dispersion $\|p\|_1$ in terms of the lemniscate length $\ell(\partial E_1(p))$, thus partially justifying Heuristic \ref{main-heuristic}(i).  This control on the dispersion in turn allows for refined estimates for the error terms $X_1,\dots,X_5$, which ultimately leads to the improvement Theorem \ref{main-thm}(ii).  We carry out these arguments in Section \ref{main-ii-sec}.

To push beyond the $O(1)$ error term in Theorem \ref{main-thm}(ii) requires further analysis.  The previous arguments allow one to show that in all remaining cases, the dispersion $\|p\|_1$ is bounded; elementary arguments (combined with another application of Theorem \ref{stokes}) also permit one to show that the origin repulsion $\|p\|_0$ is also bounded; see Section \ref{elem-sec}.  This permits for further elementary pointwise estimates on $p$ and $p'$ (again, see Section \ref{elem-sec}) that are quite precise (especially in the interior portion of the disk $D(0,1)$), and allow for very accurate estimation of the lemniscate length $\ell(\partial E_1(p))$, both in interior, intermediate, and outer portions of that disk by direct arclength calculations (and some tools for counting zeroes of polynomials, such as Bezout's theorem and Rouche's theorem).  This eventually allows one to reduce the $O(1)$ losses in lemniscate length estimates to $O(\eps)$ for any fixed $\eps>0$, leading to the proof of Theorem \ref{main-thm}(iii). We carry out these arguments in Section \ref{main-iii-sec}.

The final part (iv) of Theorem \ref{main-thm} is achieved by a similar strategy.  Firstly, the analysis used to establish part (iii) is pushed further to show that the total size $\|p\|$ is not merely bounded, but is in fact small.  A finer inspection of the various defect inequalities available then shows that one can in fact improve the upper bound on lemniscate length by a small multiple of $\|p\|$, as per Heuristic \ref{main-heuristic}; there are some losses incurred when doing so, but one can contain these losses to be of size $O(\|p\|/C_0)$ or better for a large constant $C_0$, which lets us show that $\|p\|$ is not only small, but in fact vanishes.  We carry out these arguments in Section \ref{main-iv-sec}.

\subsection{Organization of the paper}

The proofs of the four components of Theorem \ref{main-thm} share many common ingredients.  In order to keep the paper at a manageable length, we have therefore structured the paper to first introduce some general classes of tools for controlling various aspects of the geometry of lemniscates, and then use these tools to establish each of the four components of Theorem \ref{main-thm} in turn. More precisely:

\begin{itemize}
    \item In Section \ref{notation-sec} we lay out our basic notation and easy observations.
    \item In Section \ref{arclength-sec} we establish some formulae for the arclength of a lemniscate, and use this to obtain precise control on the standard lemniscate $\partial E_1(p_0)$.
    \item In Section \ref{area-sec} we establish some area theorems that give reasonable control on the ``roundness'' of regions such as $E_2(p)$ or $E_4(p)$, as well as some preliminary $L^2$ control on $p'$.
    \item In Section \ref{riesz-sec} we use rearrangement inequalities to control Riesz-type potentials, which are particularly useful to control the various expressions appearing in Theorem \ref{stokes}.
    \item In Section \ref{elem-sec} we establish some elementary pointwise estimates on $p$ and $p'$ that are particularly useful in the regime when $\|p\|$ is small.
    \item In Section \ref{defect-sec} we investigate the defect in the triangle inequality \eqref{triangle}, and use this to obtain upper and lower bounds on $\Psi(E)$ for various sets $E$.
    \item In Section \ref{stokes-sec} we establish the crucial estimate, Theorem \ref{stokes}, through an application of Stokes' theorem combined with an integration by parts argument.
    \item In Section \ref{main-i-sec} we use the tools from the previous sections (particularly those from Sections \ref{area-sec}, \ref{riesz-sec}, \ref{stokes-sec}) to give a very short proof of Theorem \ref{main-thm}(i).
    \item In Section \ref{main-ii-sec} we use the tools from the previous sections (particularly those from Sections \ref{area-sec}, \ref{riesz-sec}, \ref{defect-sec}, \ref{stokes-sec}) to obtain the improved bound of Theorem \ref{main-thm}(ii).
    \item In Section \ref{main-iii-sec} we use the tools from all of the previous sections to give a (significantly more complicated) proof of the further improvement in Theorem \ref{main-thm}(iii).
    \item Finally, in Section \ref{main-iv-sec} we use the tools from all of the previous sections (and a repetition of the arguments in Section \ref{main-iii-sec}) to prove the final result, Theorem \ref{main-thm}(iv).
\end{itemize}

It is possible that the proofs of Theorem \ref{main-thm}(iii), (iv) could be merged together; however the combined proof would then be extremely complicated, and so we have elected to prove them separately, even if this causes some of the arguments to be repeated.

\subsection{Acknowledgments}

The author learned about this problem from the Erd\H{o}s problem site \cite{bloom}.

Various AI tools, including Gemini, Gemini DeepResearch, ChatGPT Pro, and AlphaEvolve, were used to perform literature searches and provide numerical experiments, proofs and verifications of individual claims in the paper, and initial code for the figures used.  However, the final arguments in this paper are human-generated.

The author was supported by the James and Carol Collins Chair, the Mathematical Analysis \& Application Research Fund, and by NSF grant DMS-2347850, and is particularly grateful to recent donors to the Research Fund.

\section{Notation}\label{notation-sec}

\subsection{Asymptotic notation}

We use the following asymptotic notation conventions:
\begin{itemize}
    \item We use $X \ll Y$, $Y \gg X$, or $X = O(Y)$ to denote an estimate of the form estimate $|X| \leq C Y$ for some absolute constant $C>0$.  If this implied constant $C$ needs to depend on additional parameters, we denote this by subscripts; for instance $X \ll_\eps Y$ denotes an estimate of the form $|X| \leq C_\eps Y$ where $C_\eps$ can depend on $\eps$.
    \item We use $X \asymp Y$ as an abbreviation for $X \ll Y \ll X$.
    \item The degree $n$ is viewed as an asymptotic parameter going to infinity, and we will always assume $n>2$ to avoid some minor degeneracies arising from the $n=1,2$ cases.  We use $X = o(Y)$ to denote a bound of the form $|X| \leq c(n) Y$ where $c(n) \to 0$ as $n \to \infty$.  Here we allow the decay rate $c(n)$ to depend on other parameters such as $\eps$ that are declared to be fixed as $n \to \infty$.
\end{itemize}

\subsection{The geometry of sets in the plane}

For any complex number $z_0$ and radius $r>0$, we use
\begin{align*}
D(z_0,r) &\coloneqq \{ z \in \C: |z-z_0| < r \}\\
\overline{D(z_0,r)} &\coloneqq \{ z \in \C: |z-z_0| \leq r \}\\
\partial D(z_0,r) &\coloneqq \{ z \in \C: |z-z_0| = r \}
\end{align*}
to denote the open disk, closed disk, and circle respectively of radius $r$ centered at $z_0$.  For $r_1 < r_2$, we also define the annulus
$$ \Ann(z_0,r_1,r_2) \coloneqq D(z_0,r_2) \backslash D(z_0,r_1) = \{ z \in \C: r_1 \leq |z-z_0| < r_2 \}.$$

If $E$ is a measurable subset of $\C$, we use $|E| = \int_E\ dA$ to denote its area, and $1_E$ to denote its indicator function, thus $1_E(z)$ equals $1$ when $z \in E$ and $0$ otherwise.  We use
\begin{equation}\label{rE-def}
    \r[E] \coloneqq (|E|/\pi)^{1/2}
\end{equation}
to denote the \emph{area-equivalent radius} of a disk with the same area as $E$, thus
\begin{equation}\label{E-area}
    |E| = \pi \r[E]^2 = |D(z_0,\r[E])|
\end{equation}
for all $z_0 \in \C$.  This radius $\r[E]$ will be a usefully normalized way of measuring the area of $E$.  For instance, the P\'olya inequality \eqref{polya} can be rewritten as
\begin{equation}\label{polya-r}
    \r[E_r(p)] \leq \r[E_r(z^n)] = r^{1/n}.
\end{equation}

If $A$ is a finite set, we use $\# A$ to denote its cardinality.

Given a non-empty multiset $\{z_1,\dots,z_n\}$ of complex numbers $z_1,\dots,z_n$, we define their \emph{mean}
$$ \bar{z} \coloneqq \frac{1}{n} \sum_{i=1}^n z_i$$
and \emph{$\ell^1$ dispersion}
\begin{equation}\label{disp-def}
    \Disp \{ z_1,\dots,z_n \} \coloneqq \sum_{i=1}^n |z_i - \bar{z}|
\end{equation}
This will be our primary way of measuring how ``concentrated'' a set of complex numbers are around their mean.
Observe from the triangle inequality that
\begin{equation}\label{disp-upper}
 |z_0 - \bar{z}| \leq \frac{1}{n} \sum_{i=1}^n |z_0 - z_i| \leq |z_0 - \bar{z}| + \frac{1}{n} \Disp \{z_1,\dots,z_n\}
\end{equation}
for any complex number $z_0$. Hence by further application of the triangle inequality, we have
\begin{equation}\label{disp}
 \Disp \{ z_1,\dots,z_n\} \leq n |z_0 - \bar{z}| + \sum_{i=1}^n |z_i - z_0| \leq 2 \sum_{i=1}^n |z_i - z_0|.
\end{equation}
Obviously, this implies that
\begin{equation}\label{disp-sum}
 \Disp \{ z_1,\dots,z_n\} \leq \frac{2}{\# A} \sum_{\xi \in A} \sum_{i=1}^n |z_i - \xi|
\end{equation}
for any finite non-empty set $A$ of complex numbers.

\subsection{Complex differentiation}

If $z$ is a non-zero complex number, we use $\log z \in \C / 2\pi i\Z$ to denote the (multi-valued) complex logarithm of $z$, and $\arg z = \Im \log z \in \R/2\pi\Z$ to denote the (multi-valued) argument, thus
$$ z = |z| e^{i \arg z}.$$
We use $\partial$ to denote the usual Wirtinger derivative
$$ \partial \coloneqq \frac{1}{2} \frac{\partial}{\partial x} - \frac{i}{2} \frac{\partial}{\partial y}.$$
Thus, for any meromorphic function $f$ and any smooth function $\eta \colon \R \to \R$, one sees from the Cauchy--Riemann equations and the usual laws of differentiation that
\begin{align}
    \partial f &= f' \label{wirt-1}\\
    \partial \overline{f} &= 0 \label{wirt-2}\\
    \partial (|f|^2) = \partial (\overline{f} f) &= \overline{f} f' \label{wirt-3}\\
    \partial |f| = \frac{\partial (|f|^2)}{2|f|} &= \frac{\overline{f} f'}{2|f|} \label{wirt-4}\\
    \partial \eta(|f|) &= \eta'(|f|) \frac{\overline{f} f'}{2|f|} \label{wirt-5} \\
    \partial \frac{|f|}{f} &= - \frac{|f|}{f} \frac{f'}{2f} \label{wirt-6} \\
    \partial \log f &= \frac{f'}{f} \label{wirt-7}\\
    \partial \overline{\log f} &= 0 \label{wirt-8}\\
    \partial \arg f = \partial \frac{\log f - \overline{\log f}}{2i} &= \frac{1}{2i} \frac{f'}{f} \label{wirt-9}
\end{align}
away from the poles and zeroes of $f$, and choosing suitable branches of the logarithm and argument as needed.

\subsection{Critical points and partial fraction expansions}\label{disp-sec}

Let $p$ be a normalized maximizing polynomial of degree $n$. Throughout this paper, summations or set builder notation over $\zeta$ (or $\zeta'$) are assumed to range over the $n-1$ critical points $p'(\zeta)=0$ of $p$, counting multiplicity.  Thus for example $\{\zeta\}$ is the multiset of critical points, while $\{\zeta: \zeta \in D(z_0,r)\}$ denotes the multiset of critical points in the disk $D(z_0,r)$.  Since $p$ has vanishing $z^{n-1}$ coefficient, $p'$ has vanishing $z^{n-2}$ coefficients, hence from \eqref{pp-factor} the critical points $\zeta$ have mean zero:
\begin{equation}\label{zeta-mean}
\bar \zeta \coloneqq \frac{1}{n-1} \sum_\zeta \zeta = 0.
\end{equation}
In particular, comparing \eqref{disp-def} and \eqref{semi-norm} we have
\begin{equation}\label{disp-zeta}
    \Disp \{\zeta\} = \sum_\zeta |\zeta| = \|p\|_1.
\end{equation}
We also record a variant:

\begin{lemma}[Splitting the dispersion]\label{disp-split} If $r \geq 2\|p\|_1/n$, then
\begin{equation}\label{disp-zeta-r}
    \|p\|_1 \ll \sum_{\zeta: \zeta \notin D(0,r)} |\zeta| + n \Disp \{\zeta: \zeta \in D(0,r)\}.
\end{equation}
\end{lemma}

\begin{proof}  From \eqref{zeta-mean} and the triangle inequality we have
$$ \left|\sum_{\zeta: \zeta \in D(0,r)} \zeta\right| = \left|\sum_{\zeta: \zeta \notin D(0,r)} \zeta\right| \leq \sum_{\zeta: \zeta \notin D(0,r)} |\zeta|,$$
while from \eqref{markov-many} we have $\gg n$ critical points $\zeta$ in $D(0,r)$.  Hence the mean of these critical points is $O( \sum_{\zeta: \zeta \notin D(0,r)} |\zeta| / n )$.  From \eqref{disp-upper} we conclude that
$$ \sum_{\zeta: \zeta \in D(0,r)} |\zeta| \ll \sum_{\zeta: \zeta \notin D(0,r)} |\zeta| + \Disp \{\zeta: \zeta \in D(0,r)\},$$
and the claim follows from \eqref{semi-norm}.
\end{proof}

Similarly to \eqref{psi-expand}, we also have
\begin{equation}\label{phi-expand}
    \varphi(z) = \sum_{p(\eta)=0} \frac{1}{z-\eta}
\end{equation}
where $\eta$ ranges over the $n$ zeroes of $p$, counting multiplicity, as well as the variant formulae
\begin{equation}\label{log-phi-expand}
    \frac{\varphi'(z)}{\varphi(z)} = \psi(z) - \phi(z) = \sum_\zeta \frac{1}{z-\zeta} - \sum_{p(\eta)=0} \frac{1}{z-\eta}
\end{equation}
and
\begin{equation}\label{log-psi-expand}
    \frac{\psi'(z)}{\psi(z)} = \sum_{p''(\xi)=0} \frac{1}{z-\xi} - \sum_{\zeta} \frac{1}{z-\zeta}
\end{equation}
where $\xi$ ranges over the $n-2$ roots of $p''$, counting multiplicity.

\subsection{Conjugation}\label{conj-sec}

Given a polynomial $p$ of degree $n$, we define its \emph{conjugate} $\tilde p$ to be the polynomial
\begin{equation}\label{tildep-def}
    \tilde p(z) \coloneqq \overline{p(\overline{z})},
\end{equation}
that is to say the polynomial formed by applying complex conjugation to each of the coefficients of $p$.  Note that if $z$ is on the circle $\partial D(0,r)$ for some $r>0$, we have $\overline{z} = r^2/z$, and hence the condition $|p(z)|=1$ is equivalent on this circle to the degree $2n$ equation
\begin{equation}\label{tilde-r}
z^n p(z) \tilde p(r^2/z) - z^n = 0.
\end{equation}
In particular, we see that the lemniscate $\partial E_1(p)$ either intersects the circle $\partial D(0,r)$ in at most $2n$ points, or contains the entire circle; this can also be seen from Bezout's theorem, and was also noted explicitly in \cite[Lemma 1]{eremenko}.

\section{Arclength formulae}\label{arclength-sec}

One obvious way to compute the length of lemniscates $\ell(\partial E_1(p))$ is by parameterizing the curves that comprise the lemniscate.  A potential difficulty is caused by critical points $\zeta$, around which the lemniscate may self-intersect or have cusps.  However, provided one restricts to open sets that avoid the critical points, we have explicit formulae\footnote{See \cite[Theorem 1]{kosukhin} for a version of \eqref{arcl-1} that can handle critical points within $\Omega$.}:

\begin{lemma}[Arclength formulae]\label{arclength}  Let $p$ be a non-zero polynomial, and let $\Omega$ be a semialgebraic open set that does not contain any critical points of $p$ in its closure.
\begin{itemize}
    \item[(i)] (First arclength formuula) We have have the arclength formula
\begin{equation}\label{arcl-1}
 \ell(\partial E_1(p) \cap \Omega) = \int_{-\pi}^\pi \sum_{z \in \Omega: p(z) = e^{i\alpha}} \frac{1}{|p'(z)|}\ d\alpha.
\end{equation}
Note that the sum over $z$ is finite by the fundamental theorem of algebra, and is multiplicity-free since $\Omega$ avoids critical points.
\item[(ii)] (Second arclength formula)  If in addition $\Omega$ avoids the origin $0$, and we have the transversality condition
\begin{equation}\label{radial}
    \arg \frac{p(z)}{zp'(z)} \neq 0 \pmod \pi
\end{equation}
for all $z \in \partial E_1(p) \cap \Omega$,
then we have
\begin{equation}\label{arcl-2}
 \ell(\partial E_1(p) \cap \Omega) = \int_0^\infty \sum_{z \in \partial E_1(p) \cap \Omega \cap \partial D(0,r)} \frac{1}{\left|\sin \arg \frac{p(z)}{zp'(z)}\right|} \ dr,
\end{equation}
where the sum over $z$ is again finite and multiplicity-free. In particular, we have the lower bound
$$ \ell(\partial E_1(p) \cap \Omega) \geq \int_0^\infty \# (\partial E_1(p) \cap \Omega \cap \partial D(0,r)) \ dr.$$
\end{itemize}
\end{lemma}

In practice\footnote{In principle, one could use other one-dimensional arclength formulae as well, such as the Crofton formula; while versions of this formula have been useful in other literature on this problem (see, e.g., \cite[\S 5]{fryntov}), we were unable to make good use of it here.  On the other hand, we will certainly make use of two-dimensional arclength formulae arising from Stokes' theorem; see Section \ref{stokes-sec}.}, the first arclength formula \eqref{arcl-1} will be useful to control outer regions $\partial E_1(p) \cap \Ann(0,r_+,1+\sigma)$ of the lemniscate when $\|p\|$ is small and $r_+, 1+\sigma$ are reasonably close to $1$, while the second arclength formula \eqref{arcl-2} is useful to control intermediate regions $\partial E_1(p) \cap \Ann(r_-, r_+)$ of the lemniscate when $\|p\|$ is small and $r_-, r_+$ are bounded away from both $0$ and $1$, though in both cases one should remove exceptional sets corresponding to regions that are too close to a critical point.  For future reference we observe that
\begin{equation}\label{cosec-asym}
\frac{1}{\left|\sin \arg \frac{p(z)}{zp'(z)}\right|} = 1 + O\left( \dist\left(\arg \frac{p(z)}{zp'(z)},\frac{\pi}{2} + \pi \Z \right)^2 \right)
\end{equation}
if $\arg \frac{p(z)}{zp'(z)}$ is bounded away from $\pi \Z$.

\begin{proof} By the hypotheses on $\Omega$, the partial lemniscate $\partial E_1(p) \cap \Omega$ is smooth, semi-algebraic, and one-dimensional, and is thus the union of a finite number $\gamma_1,\dots,\gamma_m$ of smooth curves.  The polynomial $p$ maps each curve $\gamma_j$ to the unit circle by a local diffeomorphism; by decomposing the curves (and removing some endpoints), we may assume that each curve maps to an arc $\{ e^{i\alpha}: \alpha \in I_j \}$ for some open interval $I_j \subset (-\pi,\pi)$.  Making the change of variables $p(z) = e^{i\alpha}$, we obtain
$$ \ell(\gamma_j) = \int_{I_j} \frac{1}{|p'(z)|}\ d\alpha,$$
and (i) follows by summing over $j$.

Now we prove (ii). If $z \in \gamma_j$ is non-zero and $v$ is a tangent vector to $\gamma_j$ at $z$, then $p'(z) v$ must be tangent to the unit circle at $p(z)$, thus $\frac{p'(z) v}{p(z)}$ is imaginary.  Assuming \eqref{radial}, this shows that the tangent vector $v$ cannot be orthogonal to $z$.  By the inverse function theorem, we can therefore parameterize $\gamma_j$ in polar coordinates as $z = r e^{i\theta_j(r)}$ for some smooth function $\theta_j \colon J_j \to \R$ on some open interval $J_j \subset (0,\infty)$.  Differentiating the identity
$$ |p(r e^{i\theta_j(r)})|^2 = 1$$
in $r$, we see that
$$ \Re\left[ \overline{p}(re^{i\theta_j(r)}) p'(re^{i\theta_j}(r)) (1 + i r \theta'_j(r)) e^{i\theta_j(r)} \right] = 0$$
which on taking arguments and writing $z = re^{i\theta_j(r)}$ gives
$$ \arctan(r \theta'_j(r)) - \arg \frac{p(z)}{z p'(z)} = \pm \frac{\pi}{2} \pmod{\pi},$$
and hence
$$ \theta'_j(r) = \left(r \tan \arg \frac{p(z)}{z p'(z)}\right)^{-1}.$$
Reparameterizing $\gamma_j$ by $r$, we conclude that
\begin{align*}
     \ell(\gamma_j) &= \int_{J_j} |1 + ir \theta'_j(r)|\ dr\\
     &= \int_{J_j} \frac{1}{\left|\sin \arg \frac{p(z)}{z p'(z)}\right|} \ dr,
\end{align*}
and the claim (ii) follows by summing over $j$.
\end{proof}

We illustrate these formulae in the model case $p=p_0(z) = z^n - 1$.

\begin{lemma}[The lemniscate for $p_0$]\label{out}  We have
\begin{equation}\label{lbo-ident}
 \ell(\partial E_1(p_0)) = \int_{-\pi}^\pi |1+e^{i\alpha}|^{-\frac{n-1}{n}}\ d\alpha = B\left(\frac{1}{2},\frac{1}{2n}\right)
\end{equation}
and
\begin{equation}\label{lbo-in} \ell(\partial E_1(p_0) \cap D(0,r_0)) = 2nr_0 + O(r_0^{2n+1})
\end{equation}
for any $0 < r_0 \leq 1$.  In particular,
\begin{equation}\label{lout}
    \ell(\partial E_1(p_0) \backslash D(0,r_0)) = \int_{I_{r_0}} |1 + e^{i\alpha}|^{-\frac{n-1}{n}} \ d\alpha = \ell(\partial E_1(p_0)) - 2nr_0 + O( r_0^{2n+1})
\end{equation}
where $I_{r_0}$ is the set of all $-\pi < \alpha < \pi$ with $|1+e^{i\alpha}|^{1/n} \geq r_0$.
\end{lemma}

\begin{proof}  Substituting $p = p_0$ into \eqref{arcl-1} with $\Omega = \C \backslash D(0,\eps)$, and then sending $\eps \to 0$, we obtain
$$ \ell(\partial E_1(p_0)) = \int_{-\pi}^\pi \sum_{z: z^n = 1+e^{i\alpha}} \frac{1}{n|z|^{n-1}}\ d\alpha.$$
For each $\alpha \in (-\pi,\pi)$, there are $n$ roots to the equation $z^n = 1+e^{i\alpha}$, each with $|z|^{n-1} = |1+e^{i\alpha}|^{\frac{n-1}{n}}$.  This gives the first part of \eqref{lbo-ident}.  For the second part, we make the change of variables $t = \alpha/2$ and note that $|1+e^{i\alpha}| = 2 \cos t$, and use symmetry, to write
$$ \int_{-\pi}^\pi |1+e^{i\alpha}|^{-\frac{n-1}{n}}\ d\alpha = 2^{\frac{n+1}{n}} \int_0^{\pi/2} (\cos t)^{-\frac{n-1}{n}}\ dt.$$
The claim \eqref{lbo-ident} now follows from the standard beta function identity
$$ B(p,q) = 2 \int_0^{\pi/2} (\sin t)^{2p-1} (\cos t)^{2q-1}\ dt$$
for $p,q > 0$.

To prove \eqref{lbo-in}, we use \eqref{arcl-2} with $\Omega = D(0,r_0) \backslash D(0,\eps)$ (and then send $\eps \to 0$ as before).  If $z \in \partial D(0,r)$ for some $0 < r < r_0$ non-zero with $p_0(z) = e^{i\alpha}$, then from the identity $z p'_0(z) = n(1 + p_0(z))$ we have
$$ \arg \frac{p_0(z)}{zp_0'(z)} = \arg \frac{e^{i\alpha}}{1 + e^{i\alpha}} = \frac{\alpha}{2} \mod \pi$$
ensuring \eqref{arcl-2} (since $r < r_0 \leq 1$ forces $\alpha$ bounded away from zero); then by \eqref{cosec-asym} we have
$$ \frac{1}{\left|\sin \arg \frac{p_0(z)}{zp_0'(z)}\right|} = 1 + O((\pi - |\alpha|)^2) = 1 + O(|1 + e^{i\alpha}|^2) = 1 + O(r^{2n}).$$
From \eqref{arcl-2} and a perturbative analysis near the origin we also see that $\partial E_1(p_0) \cap D(0,r_0)$ consists of $2n$ disjoint curves emenating from the origin to the boundary $\partial D(0,r_0)$ that intersect each intermediate circle $\partial D(0,r)$ transversally; thus $\partial E_1$ intersects $\partial D(0,r)$ in exactly $2n$ points for $0 < r < r_0$.  We conclude from \eqref{arcl-2} that
$$
 \ell(\partial E_1(p) \cap \Omega) = \int_0^{r_0} n \left( 1 + O( r^{2n} ) \right)\ dr$$
 and the claim \eqref{lbo-in}  follows.  Finally, the first identity in \eqref{lout} follows from a further appeal to \eqref{arcl-1}, and the second identity then follows by subtracting \eqref{lbo-in} from \eqref{lbo-ident}.
\end{proof}

\section{Area theorems}\label{area-sec}

In this section we record some consequences of the classical area theorems of complex analysis for a normalized maximizing polynomial $p$ of degree $n$.  These theorems will not be strong in the regime where $\|p\|$ is bounded or small, but are valid even when $\|p\|$ is large, and will be useful for controlling various error terms, as well as placing regions such as $E_2$ or $E_4$ in a slight enlargement of the standard unit disk $D(0,1)$.

For any $R>0$, the map $p$ is an $n$-to-one holomorphic map from $E_R(p)$ onto $D(0,R)$, and hence by the change of variables formula
\begin{equation}\label{area-form}
    \int_{E_R(p)} |p'|^2\ dA = n \int_{D(0,R)}\ dA = \pi n R^2.
\end{equation}

By definition of normalized maximizing polynomial, the complement of $\overline{E_1(p)}$ in the Riemann sphere $\C \cup \{\infty\}$ is simply connected, and hence by the Riemann mapping theorem (and the Schwartz lemma) there exists a unique conformal map with a Laurent series expansion
$$ z = a_{-1} w + a_0 + \frac{a_1}{w} + \frac{a_2}{w^2} + \dots$$
from $\C \backslash \overline{D(0,1)}$ to $\C \backslash \overline{E_1(p)}$, with $a_{-1}$ a positive real (the logarithmic capacity of $E_1(p)$).  Indeed, this map must equal the inverse of a branch of $p(z)^{1/n}$ on $\C \backslash E_1(p)$.  Since $p$ is normalized, we have $p(z)^{1/n} = z + O(1/z)$ as $|z| \to \infty$, and hence $a_{-1}=1$ and $a_0=0$.  We conclude the identity
\begin{equation}\label{p-power}
 p\left( w + \sum_{k=1}^\infty \frac{a_k}{w^k} \right)^n = w
\end{equation}
for all $w \in \C \backslash \overline{E_1(p)}$.  In particular, for $R>1$, $E_R(p)$ is connected, whose boundary curve $\partial E_R(p)$ is simple, smooth, and parameterized by the formula
\begin{equation}\label{param}
\partial E_R(p) = \left\{ R^{1/n} e^{i\theta} + \sum_{k=1}^\infty \frac{a_k}{R^{k/n} e^{in\theta}}: 0 \leq \theta\leq 2\pi\right\}.
\end{equation}
A standard application of Stokes' theorem then gives the familiar Gr\"onwall area formula
\begin{equation}\label{area-formula}
    |E_R| = \pi R^{2/n} - \pi \sum_{k=1}^\infty \frac{k|a_k|^2}{R^{2k/n}}
\end{equation}
for $R>1$, which is of course consistent with \eqref{polya}, but describes precisely the defect in this inequality.  From \eqref{rE-def} we thus have
\begin{equation}\label{area-formula-alt}
    \r[E_R] = \left(R^{2/n} - \sum_{k=1}^\infty \frac{k |a_k|^2}{R^{2k/n}}\right)^{1/2}
\end{equation}
or equivalently
\begin{equation}\label{area-formula-2}
\sum_{k=1}^\infty \frac{k |a_k|^2}{R^{2k/n}} = R^{2/n} - \r[E_R]^2.
\end{equation}
We also obtain the arclength formula
\begin{equation}\label{arc}
\ell(\partial E_R(p)) = \int_0^{2\pi} \left| R^{1/n} i e^{i\theta} - \sum_{k=1}^\infty \frac{ik a_k}{R^{k/n} e^{ik\theta}} \right|\ d\theta,
\end{equation}
By Cauchy--Schwarz and Plancherel we thus have
\begin{equation}\label{arc-bound}
\begin{split}
    \ell(\partial E_R(p)) &\leq \left( 2\pi \int_0^{2\pi} \left| R^{1/n} i e^{i\theta} - \sum_{k=1}^\infty \frac{ik a_k}{R^{k/n} e^{ik\theta}} \right|^2\ d\theta\right)^{1/2} \\
    &= 2\pi \left( R^{2/n} + \sum_k \frac{k^2 |a_k|^2}{R^{2k/n}} \right)^{1/2}.
\end{split}
\end{equation}
In practice, this will mean that $\partial E_R(p)$ will be considerably shorter than $\partial E_1(p)$ for say $R=4$. Indeed, $E_R(p)$ typically becomes significantly ``rounder'' and more disk-like as $R$ increases: see Figure \ref{fig:lemniscates}.

\begin{figure}[t]
\centering
\includegraphics[width=5in]{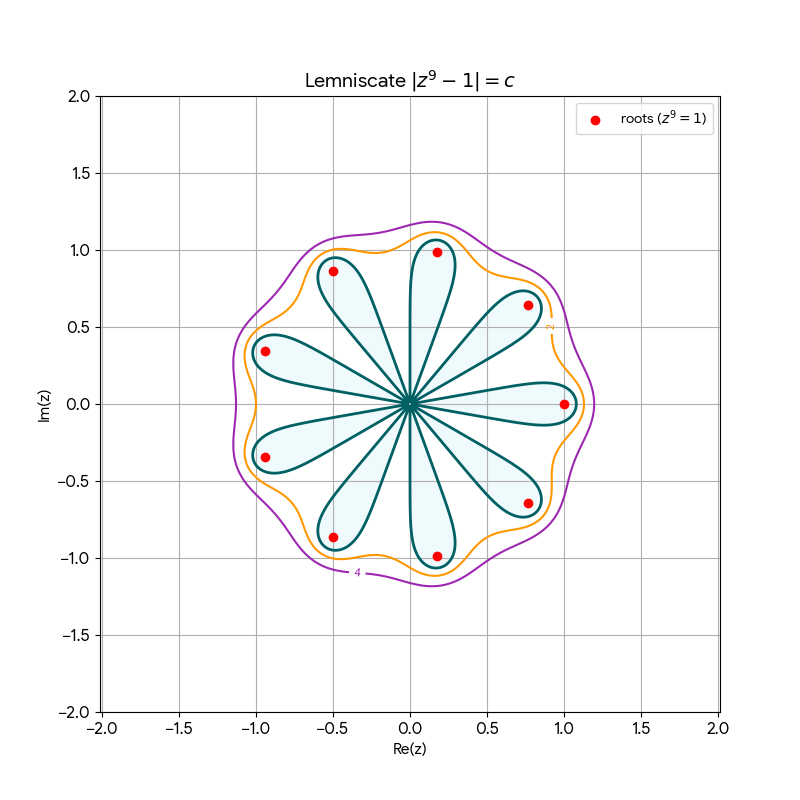}
\caption{The lemniscates $\partial E_R(p_0)$ for $n=9$ and $R=1,2,4$.  Note how much closer $E_2(p_0)$ and $E_4(p_0)$ are to disks than $E_1(p_0)$.}
\label{fig:lemniscates}
\end{figure}

From \eqref{param}, the triangle inequality, and Rouche's theorem we also obtain the inclusions
\begin{equation}\label{inclusions}
    D\left(0, R^{1/n} - \sum_{k=1}^\infty \frac{|a_k|}{R^{k/n}}\right) \subset E_R(p) \subset \overline{D\left(0, R^{1/n} + \sum_{k=1}^\infty \frac{|a_k|}{R^{k/n}}\right)}
\end{equation}
(with the former inclusion vacuous if the radius is zero or negative).  In particular, since the lemniscate contains all the critical points, one has
\begin{equation}\label{critical}
    |\zeta| \leq R^{1/n} + \sum_{k=1}^\infty \frac{|a_k|}{R^{k/n}}
\end{equation}
for all critical points $\zeta$ (this can also be obtained from the Gauss--Lucas theorem).

\section{Riesz potentials}\label{riesz-sec}

Given a bounded measurable set $E \subset \C$, define the \emph{Riesz potential} $I_1 1_E \colon \C \to \R$ to be the function
$$ I_1 1_E(z_0) \coloneqq \int_{E} \frac{1}{|z-z_0|}\ dA(z)$$
where we recall that  $dA$ is area measure.  In this section we record various basic properties of this potential.

The Hardy-Littlewood rearrangement inequality \cite[Theorems 368--370, 378]{hl} implies that for any bounded measurable set $E$, one has
\begin{equation}\label{riesz-bound}
I_1 1_E(z_0) \leq I_1 1_{D(0, \r[E])}(0) = 2 \pi \r[E]
\end{equation}
for all $z_0$, where $\r[E]$ was defined in \eqref{rE-def}.  In particular, for any complex numbers $\zeta_1,\dots,\zeta_m$, one has
\begin{equation}\label{multip}
     \int_E \sum_{j=1}^m \frac{1}{|z-\zeta_j|}\ dA(z) \leq 2 \pi m \r[E]
\end{equation}
for all $E$; applying this to a superlevel set $E = \{ z:\sum_{j=1}^m \frac{1}{|z-\zeta_j|} \geq \lambda \}$ and rearranging, one obtains the distributional bound
\begin{equation}\label{distrib}
 \left|\left\{ z :\sum_{j=1}^m \frac{1}{|z-\zeta_j|} \geq \lambda \right\}\right| \leq \frac{4\pi m^2}{\lambda^2}
\end{equation}
for any $\lambda>0$.  (The constant $4\pi$ here can likely be improved, but we will not need to do so here.)

\begin{remark}  One can obtain defect versions of the inequalities \eqref{riesz-bound}, \eqref{multip}, \eqref{distrib}.  In the case of \eqref{multip}, for instance, it is possible to extract additional gains proportional to the $\ell^2$ dispersion $\|p\|_2$ (assuming $p$ is normalized).  However, as discussed in Remark \ref{l2}, any control on $\ell^2$ dispersion $\|p\|_2$ obtained by such methods will be inferior to control on $\ell^1$ dispersion $\|p\|_1$, which as we shall see in Section \ref{defect-sec} can be extracted from defect inequalities for the triangle inequality \eqref{triangle}.
\end{remark}




\section{Elementary bounds}\label{elem-sec}

Let $p$ be a monic polynomial of degree $n$ obeying \eqref{zeta-mean}.  In this section we collect some elementary estimates on $p$ in terms of the critical points $\zeta$.  It will be useful to introduce the normalized distance function
\begin{equation}\label{delta-def}
    \delta(z) \coloneqq \frac{\min_\zeta |z-\zeta|}{|z|}
\end{equation}
to the critical points, for any $z \neq 0$.  Note from \eqref{markov-many} and the triangle inequality that
\begin{equation}\label{delta-upper}
    \delta(z) \ll 1
\end{equation}
whenever $|z| \gg \|p\|_1/n$.

Observe that if the dispersion $\|p\|_1 = \sum_\zeta |\zeta|$ vanishes, then
$$ p'(z) = nz^{n-1}$$
and hence also
$$ p(z) = z^n + p(0).$$
In particular, $p$ is nearly constant near the origin.  The following proposition shows that this behavior persists when the dispersion is merely small, rather than vanishing.

\begin{proposition}[Near-constancy near origin]\label{pvar}  Let $z \in \C$.
    \begin{itemize}
        \item[(i)]  If $z \neq 0$, we have the upper bound
\begin{equation}\label{p'-bound}
    |p'(z)| \leq n |z|^{n-1} e^{O(\min(\|p\|_1/|z|, \|p\|_1^2/|z|^2))}
\end{equation}
and the lower bound
\begin{equation}\label{p'-bound-lower}
    |p'(z)| \geq n |z|^{n-1} e^{-O(\min(\|p\|_1/|z|, \|p\|_1^2/|z|^2))} \delta(z)^{O(\|p\|_1/|z|)}.
\end{equation}
and
\begin{equation}\label{poz}
     p(z) = p(0) + O(|z|^n e^{O(\min(\|p\|_1/|z|, \|p\|_1^2/|z|^2))}).
\end{equation}
If $|z| > 2\|p\|_1$, the $\delta(z)^{O(\|p\|_1/|z|)}$ factor in \eqref{p'-bound-lower} can be omitted.

\item[(ii)]  If $z = O(\|p\|_1/n)$, then
\begin{equation}\label{ppcl}
     |p'(z)| \leq e^{O(n)} n^{-n} \|p\|_1^{n-1}
\end{equation}
and
\begin{equation}\label{pocl}
     p(z) = p(0) + O(e^{O(n)} n^{-n} \|p\|_1^n).
\end{equation}
\end{itemize}
\end{proposition}

\begin{proof}  We begin with (i).  From \eqref{pp-factor} we can write
$$ |p'(z)| = n |z|^{n-1} \prod_\zeta \left| 1 - \frac{\zeta}{z} \right|.$$
In the spirit of the Weierstrass factorization theorem, we can also use \eqref{zeta-mean} to obtain the alternate factorization
$$ |p'(z)| = n |z|^{n-1} \prod_\zeta \left| \left( 1 - \frac{\zeta}{z} \right) e^{\zeta/z}  \right|.$$
From Taylor expansion we have the upper bounds
$$\left| 1 - \frac{\zeta}{z} \right| \leq e^{|\zeta|/|z|}$$
and
$$\left| \left( 1 - \frac{\zeta}{z} \right) e^{\zeta/z} \right| \leq e^{O(|\zeta|^2/|z|^2)};$$
inserting these bounds and using \eqref{l1-l2} we obtain \eqref{p'-bound}.  Similarly, for $\zeta \in D(0,|z|/2)$, we have the lower bounds
$$\left| 1 - \frac{\zeta}{z} \right| \geq e^{-O(|\zeta|/|z|)}$$
and
$$\left| \left( 1 - \frac{\zeta}{z} \right) e^{\zeta/z} \right| \geq e^{-O(|\zeta|^2/|z|^2)}$$
while for $\zeta \notin D(0,|z|/2)$, we have the variant lower bounds
$$\left| 1 - \frac{\zeta}{z} \right| \geq \delta(z) e^{-O(|\zeta|/|z|)}$$
and
$$\left| \left( 1 - \frac{\zeta}{z} \right) e^{\zeta/z} \right| \geq \delta(z) e^{-O(|\zeta|^2/|z|^2)}.$$
By \eqref{markov}, there are only $O(\|p\|_1/|z|)$ values of $\zeta$ outside of $D(0,|z|/2)$ (and none at all if $|z| > 2 \|p\|_1$), so inserting these lower bounds into the two factorizations of $|p'(z)|$ gives \eqref{p'-bound-lower}.

We skip \eqref{poz} for now and turn to (ii).  For $|z| \asymp \|p\|_1/n$ the claim \eqref{ppcl} comes from \eqref{p'-bound}; the case $|z| \ll \|p\|_1/n$ then follows from the maximum principle.  To obtain \eqref{pocl}, we integrate \eqref{ppcl}.
Finally we now turn to \eqref{poz}.  Let $C$ be a sufficiently large constant.  For $|z| \leq 2C\|p\|_1/n$, the claim follows from (ii) (after adjusting constants), so we may assume $|z| \geq 2C\|p\|_1/n$. Consider the quantity
$$ |p(z)-p(0)| - C |z|^n e^{C\|p\|_1/|z|}.$$
The radial derivative of this quantity is at most
$$ |p'(z)| - C n |z|^{n-1} e^{C\|p\|_1/|z|} + C^2 \|p\|_1|z|^{n-2} e^{C\|p\|_1/|z|}.$$
By \eqref{p'-bound}, this derivative is negative for $|z| \geq 2C\|p\|_1/n$ if $C$ is large enough.  Integrating and using (ii) we conclude that
$$ |p(z)-p(0)| - C |z|^n e^{C\|p\|_1/|z|} \ll_C e^{O_C(n)} n^{-n} \|p\|_1^n$$
for $|z| \geq 2C\|p\|_1/n$.  This gives the bound
$$ p(z) = p(0) + O(|z|^n e^{O(\|p\|_1/|z|)})$$
(after adjusting constants).  Thus, if we let $C'$ be another large constant, we have obtained \eqref{poz} when $|z| \leq 3C' \|p\|_1$.  To handle the case $|z| \geq 3C'\|p\|_1$, we work with the
$$ |p(z)-p(0)| - C' |z|^n e^{C'\|p\|_1^2/|z|^2},$$
whose radial derivative is at most
$$ |p'(z)| - C' n |z|^{n-1} e^{C'\|p\|_1^2/|z|^2} + 2(C')^2 \|p\|_1|z|^{n-2} e^{C'\|p\|_1^2/|z|^2}.$$
Again, \eqref{p'-bound} ensures that this derivative is negative for $|z| \geq 3C'\|p\|_1$ if $C'$ is large enough.  Integrating and using the previous case, we obtain the $|z| \geq 3C' \|p\|_1$ case of \eqref{poz}.
\end{proof}

\begin{corollary}[Near constancy near origin, II]\label{pvar-cor} We have the estimates
    \begin{equation}\label{pin}
 p(z) = -1 + O(e^{O(n)} n^{-n} \|p\|^n)
\end{equation}
for $z = O(\|p\|/n)$ and
\begin{equation}\label{pin-2}
    p(z) = -1 + O(|z|^n e^{O(\min(\|p\|/|z|,\|p\|^2/|z|^2))})
\end{equation}
for $z \neq 0$.
\end{corollary}

\begin{proof}  From \eqref{origin-repulse}, \eqref{size-def} one has $p(0) = -1 + O(n^{-n} \|p\|^n)$ and $\|p\|_1 \leq \|p\|$.  The claims now follow from Proposition \ref{pvar} and the triangle inequality.
\end{proof}

The next bound explains why we view $\|p\|_0$ as a measure of origin repulsion.

\begin{lemma}[Origin repulsion]\label{origin-repulsion}  We have
$$ \|p\|_0 \ll \| p \|_1 + n \operatorname{dist}(0, \partial E_1(p)).$$
In particular, by \eqref{size-def}, one has
$$ \|p\| \ll \| p \|_1 + n \operatorname{dist}(0, \partial E_1(p)).$$
\end{lemma}

Informally: if $\|p\|_0$ is the dominant contribution to the total size $\|p\|$, then the lemniscate $E_1(p)$ is repelled away from the origin by $\gg \|p\|_0/n$.  In practice, this will cause the lemniscate length to decrease by an amount proportional to $\|p\|_0$, as per Heuristic \ref{main-heuristic}.

\begin{proof}  Let $z$ be the closest point in $\partial E_1(p)$ to the origin, thus our task is to show that $\|p\|_0 \ll \| p \|_1 + n |z|$.  If $|z| \geq \|p\|_1/n$, then from \eqref{poz}, \eqref{npz} one has
$$ n^{-n} \|p\|_0^n \ll |z|^n$$
giving the claim; if instead $|z| < \|p\|_1/n$, then from \eqref{pocl}, \eqref{npz} one has
$$ n^{-n} \|p\|_0^n \ll e^{O(n)} n^{-n} \|p\|_1^n$$
again giving the claim.
\end{proof}

For the model polynomial $p_0(z) = z^n - 1$ we have the identity
$$ p_0(z) = -1 + \frac{z p_0'(z)}{n}.$$
It turns out that we can obtain an approximate version of this identity as long as one stays away from the origin and from the critical points.

\begin{proposition}[More precise formula for $p$]\label{pform}  Suppose that $|z| \gg \|p\|$ and $\delta(z) \gg 1/n$.  Then we have
    $$ p(z) = -1 + \frac{zp'(z)}{n} \left( 1 + O\left(\frac{\|p\|^2}{n |z|^2 \delta(z)} \right) \right).$$
    In particular, from \eqref{p'-bound-lower}, we have $p(z) \neq -1$.
\end{proposition}

\begin{proof}  For any $0 < t \leq 1$, we introduce the function
    $$ f(t) \coloneqq p(tz) + 1 - \frac{tzp'(tz)}{n},$$
then our task is to establish the bound
$$ f(1) \ll \frac{|p'(z)| \|p\|^2}{n^2 |z| \delta(z)}.$$
Set $t_0 \coloneqq \frac{c\|p\|}{|z|}$ for a small constant $c>0$, then $0 < t_0 \leq 1$.
From Proposition \ref{pvar}(i) (bounding $\|p\|_1$ by $\|p\|$) and \eqref{pin-2} we have
\begin{equation}\label{1eps}
 \begin{split}
    f(t_0) &= p(t_0 z) - \frac{t_0 p'(t_0 z)}{n} \\
    &= -1 + O_c\left( (c\|p\|)^n \right)
 \end{split}
\end{equation}
and also
$$
\frac{|p'(z)| \|p\|^2}{n^2 |z| \delta(z)} \gg n^{-O(1)} |z|^{n-2} \|p\|^2 \delta(z)^{O(1)} \gg (c\|p\|)^n
$$
for $c$ small enough, by the lower bounds on $|z|$ and $\delta(z)$.  Thus, by the triangle inequality, it will suffice to show that
$$ f(1) - f(t_0) \ll_c \frac{|p'(z)| \|p\|^2}{n^2 |z| \delta(z)}.$$
From the fundamental theorem of calculus and the triangle inequality it will suffice to show that
$$ \int_{t_0}^1 |f'(t)|\ dt \ll_c \frac{|p'(z)| \|p\|^2}{n^2 |z| \delta(z)}.$$
We therefore turn to estimating $f'(t)$ for $t_0 \leq t \leq 1$.  Using \eqref{psi-expand}, \eqref{zeta-mean} we bound
\begin{align*}
 f'(t)
&= \frac{n-1}{n} zp'(tz) - \frac{tz^2 p''(tz)}{n} \\
&= \frac{zp'(tz)}{n} \left( (n-1) - \sum_\zeta \frac{tz}{tz-\zeta} \right) \\
&= - \frac{zp'(z)}{n} \frac{p'(tz)}{p'(z)} \sum_\zeta \frac{\zeta}{tz-\zeta} \\
&= - \frac{zp'(z)}{n} \frac{p'(tz)}{p'(z)} \left( \sum_\zeta \frac{\zeta}{tz-\zeta} - \frac{\zeta}{tz} \right) \\
&= \frac{p'(z)}{tn} \frac{p'(tz)}{p'(z)} \sum_\zeta \frac{\zeta^2}{tz-\zeta} \\
&\ll \frac{|p'(z)|}{tn} \left( \sum_\zeta \frac{|\zeta|^2}{|z-\zeta|} \right) \prod_{\zeta' \neq \zeta} \frac{|tz-\zeta'|}{|z-\zeta'|} \\
&\ll \frac{|p'(z)| \sum_\zeta |\zeta|^2}{n t |z| \delta(z)} \sup_\zeta \prod_{\zeta' \neq \zeta} \frac{|tz-\zeta'|}{|z-\zeta'|}
\end{align*}
and so by \eqref{l1-l2} we reduce to showing that
$$ \int_{t_0}^1 \sup_\zeta \prod_{\zeta' \neq \zeta} \frac{|tz-\zeta'|}{|z-\zeta'|}\ \frac{dt}{t} \ll_c \frac{1}{n}.$$
If $\zeta' \in D(0, c\|p\|/4)$, then from Taylor expansion we have
$$\frac{|tz-\zeta'|}{|z-\zeta'|} = t \frac{|1-\zeta'/tz|}{|1-\zeta'/z|} = t \exp(O(|\zeta'|/t|z|)) $$
while for $\zeta' \notin D(0, c\|p\|/4)$ we have
$$  \frac{|tz-\zeta'|}{|z-\zeta'|} \leq 1 + \frac{(1-t)|z|}{|z-\zeta'|} \leq 1 + O(n(1-t))$$
since $\delta(z) \gg 1/n$.  By \eqref{size-def} only $O_c(1)$ of the critical points are in the latter category.  Applying \eqref{size-def} again, we conclude that
$$ \sup_\zeta \prod_{\zeta' \neq \zeta} \frac{|tz-\zeta'|}{|z-\zeta'|} \ll t^{n-O_c(1)} (1 + O(n(1-t)))^{O(1)} \exp( O( \|p\| / t|z| )).$$
Since $t \geq t_0$, we have $\exp( O( \|p\| / t|z| )) = O_c(1)$.  Thus we reduce to showing that
$$ \int_{t_0}^1 t^{n-O_c(1)} (1 + n(1-t))^{O(1)}\ dt \ll_c \frac{1}{n}.$$
Performing the substitution $t = 1 - \frac{s}{n}$, we reduce to showing that
$$\int_0^\infty \left(1-\frac{s}{n}\right)^{n-O_c(1)} (1 + s)^{O_c(1)}\ ds \ll_c 1.$$
But we can bound $\left(1-\frac{s}{n}\right)^{n-O_c(1)} \ll_c e^{-s/2}$, giving the claim.
\end{proof}

The above bounds are mostly useful in the regime where $\|p\|$ or $\|p\|_1$ is bounded or small.  However, even when $\|p\|_1$ is somewhat large, we can still obtain some useful control on $p'$ that improves upon the Cauchy inequality:

\begin{lemma}[Improved Cauchy inequality]\label{improv}  If $D(z_0,s)$ is a disk obeying the condition
\begin{equation}\label{zr}
|z_0| + \frac{s}{10} > \frac{2 \|p\|_1}{n},
\end{equation}
the one has the improved Cauchy inequality
\begin{equation}\label{cauchy-improv}
    |p'(z_0)| \ll e^{-\frac{cns + O(\|p\|_1)}{|z_0|+s}} \frac{\sup_{z \in D(z_0,s)} |p(z)|}{s}
\end{equation}
for some absolute constant $c>0$.
\end{lemma}

Comparing \eqref{cauchy-improv} against the standard Cauchy inequality
\begin{equation}\label{standard-cauchy}
    |p'(z_0)| \leq \frac{\sup_{z \in D(z_0,s)} |p(z)|}{s},
\end{equation}
we see that the former inequality enjoys an exponential gain in the regime $s \gg \|p\|_1/n$.  This can be contrasted with the bounds in Proposition \ref{pvar} or Proposition \ref{pform}, which are significantly more precise than \eqref{cauchy-improv}, but are only effective under the condition $|z_0| \gg \|p\|_1$, which is often more restrictive in practice.

\begin{proof}  Write $z_0 = re^{i\theta}$ and $M \coloneqq \sup_{z \in D(z_0,s)} |p(z)|$.  From \eqref{standard-cauchy} we have
$$ p'((r+ts) e^{i\theta}) \ll \frac{M}{s}$$
for all $1/3 \leq t \leq 2/3$.  In particular,
$$
 \log |p'(re^{i\theta})| \leq \log \frac{M}{s} + \log \frac{|p'(re^{i\theta})|}{|p'((r+ts) e^{i\theta})|} + O(1).
$$
Applying \eqref{pp-factor} to the right-hand side, we conclude that
$$ \log |p'(re^{i\theta})| \leq \log \frac{M}{s} + \sum_\zeta \log \frac{|re^{i\theta}-\zeta|}{|(r+ts) e^{i\theta}-\zeta|} + O(1).$$
By \eqref{markov-many}, \eqref{zr} we see that $\gg n$ of the critical points $\zeta$ lie in the disk $D(0,|z_0| + s/10)$, while only $O(\frac{\|p\|_1}{|z_0|+s})$ lie outside this disk.  By elementary geometry, we see that
$$ -\log \frac{|re^{i\theta}-\zeta|}{|(r+ts) e^{i\theta}-\zeta|} \gg \frac{s}{|z_0|+s}$$
for all critical points in the former category.  We conclude that
$$ \log |p'(re^{i\theta})| \leq \log \frac{M}{s} - \frac{c n s}{|z_0|+s} + \sum_{\zeta \notin D(0,|z_0|+s/10)} \log \frac{|re^{i\theta}-\zeta|}{|(r+ts) e^{i\theta}-\zeta|} + O(1).$$
for all $1/3 \leq t \leq 2/3$ and some absolute constant $c>0$.  We average in $t$ to obtain
\begin{equation}\label{log-b}
     \log |p'(re^{i\theta})| \leq \log \frac{M}{s} - \frac{c n s}{|z_0|+s} + 3 \sum_{\zeta \notin D(0,|z_0|+s/10)} \int_{1/3}^{2/3} \log \frac{|re^{i\theta}-\zeta|}{|(r+ts) e^{i\theta}-\zeta|}\ dt + O(1).
\end{equation}
Routine computation shows that
$$ \int_{1/3}^{2/3} \log \frac{|z|}{|z + t|}\ dt \leq O(1)$$
for any complex number $z$. Indeed for $|z| \geq 2$ this is immediate from the triangle inequality, while for $|z| < 2$ we may discard the numerator and use the uniform local integrability of the logarithm on horizontal lines.  Rotating and rescaling, we conclude that the summands in \eqref{log-b} are bounded above by $O(1)$.  Thus by \eqref{markov} we have
$$ \log |p'(re^{i\theta})| \leq \log \frac{M}{s} - \frac{c n s}{|z_0|+s} + O\left(\frac{\|p\|_1}{|z_0|+s}\right) + O(1),$$
giving the claim.
\end{proof}

\section{The defect in the triangle inequality}\label{defect-sec}

The triangle inequality asserts that
$$ \left|\sum_{i=1}^n z_i\right| \leq \sum_{i=1}^n |z_i|$$
for any complex numbers $z_1,\dots,z_n$.  We can lower bound the defect in this inequality as follows:

\begin{lemma}[Defect version of triangle inequality]\label{defect}  For any complex numbers $z_1,\dots,z_n$ with $n \geq 2$, we have
$$ \sum_{i=1}^n |z_i| - \left|\sum_{i=1}^n z_i\right| \geq \frac{\lfloor n/2 \rfloor}{n(n-1)} \sum_{1 \leq i,j \leq n} |z_i| + |z_j| - |z_i+z_j|.$$
\end{lemma}

\begin{proof} By summing $z_i$ in pairs and using the triangle inequality, we have
$$ \left|\sum_{i=1}^n z_i\right| \leq \sum_{i=1}^{n/2} |z_{2i-1} + z_{2i}| $$
when $n$ is even and
$$ \left|\sum_{i=1}^n z_i\right| \leq \sum_{i=1}^{(n-1)/2} |z_{2i-1} + z_{2i}| + |z_n|$$
when $n$ is odd.  In either case, we conclude that
$$ \sum_{i=1}^n |z_i| - \left|\sum_{i=1}^n z_i\right| \geq \sum_{i=1}^{\lfloor n/2\rfloor} |z_{2i-1}| + |z_{2i}| - |z_{2i-1} + z_{2i}|.$$
Averaging over all permutations of the $z_1,\dots,z_n$, we obtain the claim.
\end{proof}

Specializing this to the expression \eqref{psi-expand}, we can lower bound the defect in \eqref{triangle}:

\begin{corollary}\label{defect-psi}  For any $z$ that is not a critical point, we have
\begin{equation}\label{def}
 \sum_\zeta \frac{1}{|z-\zeta|} - |\psi(z)| \gg \frac{1}{n} \sum_{\zeta,\zeta'} \frac{1}{|z-\zeta|} + \frac{1}{|z-\zeta'|} - \left| \frac{1}{z-\zeta} + \frac{1}{z-\zeta'}  \right|
\end{equation}
(recalling the convention that $\zeta, \zeta'$ both vary over critical points counting multiplicity).
\end{corollary}

We can control the summands on the right-hand side of \eqref{def} in an integral sense:

\begin{lemma}[Averaged defect in triangle inequality]\label{tridef}  Suppose that $\zeta \in D(z_0,r/2)$ and $\zeta' \in D(z_0, Cr)$ for some $z_0 \in \C$ and $C, r > 0$.  Then
$$ \int_{D(z_0,r)} \frac{1}{|z-\zeta|} + \frac{1}{|z-\zeta'|} - \left| \frac{1}{z-\zeta} + \frac{1}{z-\zeta'}  \right|\ dA(z) \asymp_C |\zeta-\zeta'|.$$
\end{lemma}

\begin{proof} By translation we can take $\zeta=0$, and then by complex dilation we can take $\zeta'=1$.  Then we have $|z_0| < r/2$ and $|z_0-1| < Cr$, which imply that $r \geq \frac{1}{C+1/2}$ and hence
$$ D\left(0,\frac{1}{2C+1}\right) \subset D(z_0,r).$$
Direct calculation shows that the integrand
$$ \frac{1}{|z|} + \frac{1}{|z-1|} - \left| \frac{1}{z} + \frac{1}{z-1}  \right|$$
is non-negative, absolutely integrable, and non-vanishing away from the real axis.  The claim follows.
\end{proof}

Putting this together, we can now obtain useful lower and upper bounds on the measure $\Psi$ introduced in \eqref{Psi-def}.

\begin{corollary}[Upper bound on $\Psi$]\label{defect-psi-cor}  Suppose that $r \geq 10\|p\|_1/n$.  For any measurable set $E$ containing a disk $D(z_0, r)$ and $C>0$, we have
$$ \Psi(E) \leq 2(n-1) \r[E] - c_C \Disp \{ \zeta : \zeta \in D(z_0, Cr) \}$$
for some constant $c_C>0$ depending only on $C$.
\end{corollary}

Note that this bound is consistent with Heuristic \ref{main-heuristic} and Theorem \ref{stokes}.

\begin{proof}  From \eqref{multip} one has
$$ \frac{1}{\pi} \int_{E} \sum_\zeta \frac{1}{|z-\zeta|}\ dA(z) \leq 2(n-1) \r[E].$$
Hence by \eqref{def} and the triangle inequality, it suffices to show that
$$ \int_{E} \sum_{\zeta,\zeta'} \frac{1}{|z-\zeta|} + \frac{1}{|z-\zeta'|} - \frac{1}{|z-\zeta + z-\zeta'|}\ dA(z) \gg_C n \Disp \{ \zeta : \zeta \in D(z_0, Cr) \}.$$
We may replace $E$ by the smaller set $D(z_0,r)$.  Applying Lemma \ref{tridef}, the left-hand side is
$$ \gg_C \sum_{\zeta \in D(z_0,r/2); \zeta' \in D(z_0,Cr)} |\zeta-\zeta'|.$$
By \eqref{markov-many} and the hypothesis $r \geq 10\|p\|_1/n$, we have $\gg n$ critical points $\zeta$ in $D(z_0,r/2)$.  The claim now follows from \eqref{disp-sum}.
\end{proof}

We also record the following variant:

\begin{lemma}[Lower bound on $\Psi$]\label{ankh}  For any $r>0$, we have
$$ \Psi(\Ann(0,r/2,r))\geq (n-1) r - O(\|p\|_1).$$
\end{lemma}

\begin{proof} By \eqref{psi-expand} and the triangle inequality, we have
$$ |\psi(z)| \geq \frac{n-1}{|z|} - \sum_\zeta \left| \frac{1}{z} - \frac{1}{z-\zeta}\right|.$$
Since
$$ \int_{\Ann(0,r/2,r)} \frac{n-1}{|z|}\ dA = (n-1) \pi r,$$
it thus suffices by \eqref{Psi-def}, \eqref{semi-norm} to show that
$$ \int_{\Ann(0,r/2,r)} \left| \frac{1}{z} - \frac{1}{z-\zeta}\right|\ dA(z) \ll |\zeta|$$
for all $\zeta$.  For $|\zeta| \geq r/4$ this follows from \eqref{riesz-bound} and the triangle inequality, while for $|\zeta| < r/4$ it follows from the observation that the integrand is pointwise bounded by $O(|\zeta|/r^2)$.
\end{proof}

Finally, we record a simple lower bound:

\begin{lemma}[Lower bound]\label{lower-triangle}  If $z_1,\dots,z_m$ are nonzero complex numbers with $\arg z_i \in I \pmod{2\pi}$ for all $i=1,\dots,m$ and some interval $I$ of length at most $\pi - \alpha$, then
$$ \left| \sum_{i=1}^m z_i \right| \geq \sin(\alpha/2) \sum_{i=1}^m |z_i|.$$
\end{lemma}

\begin{proof} By rotating we may take $I$ to be a subset of $[-(\pi-\alpha)/2, (\pi-\alpha)/2]$.  We then have
    $$ \Re z_i \geq \sin(\alpha/2) |z_i|$$
for all $i=1,\dots,m$, and the claim follows by summing in $i$.
\end{proof}

\section{An application of Stokes' theorem}\label{stokes-sec}

In this section we prove Theorem \ref{stokes}. Our arguments here are inspired by the calculations in \cite[\S 4,9]{fryntov}, but use smooth cutoffs in place of the localization to squares that is performed in \cite[\S 9]{fryntov}.  We abbreviate $E_R(p)$ as $E_R$.

In \cite[\S 4]{fryntov} it is observed that one has the Stokes identity
$$ \Re \int_\gamma \frac{\overline{s(z)}\ dz}{i} = 2 \Re \int_E \partial s\ dA$$
whenever $\gamma$ is a closed piecewise smooth curve oriented anticlockwise around a region $E$, and $s$ is a locally bounded function whose derivative $\partial s$ is \emph{quasismooth} in the sense that it is smooth away from a finite set of singularities and whose gradient is locally integrable.  In particular, we have (see \cite[(4)]{fryntov})
\begin{equation}\label{stokes-ident}
     \ell(\partial E_1) = 2 \Re \int_{E_1} \partial s\ dA
\end{equation}
whenever $s$ is bounded, has quasismooth derivative, and coincides on $\partial E_1$ with the outward unit normal to $E_1$ outside of critical points.  This outward unit normal can be computed to be $|\varphi|/\varphi$, where $\varphi \coloneqq p'/p$ was introduced in \eqref{phi-def}.

If we integrate on $E_1 \cap \Omega$ for some semi-algebraic open set $\Omega$, and assume that $s$ is bounded by $O(1)$ on $E_1 \cap \overline{\Omega}$, then the boundary $\partial E_1 \cap \Omega$, being a semi-algebraic curve, is piecewise smooth, and we obtain the variant
\begin{equation}\label{stokes-ident-2}
     \ell(\partial E_1 \cap \overline{\Omega}) = 2 \Re \int_{E_1 \cap \Omega} \partial s\ dA + O(X_5)
\end{equation}
thanks to \eqref{X5-def}.

In \cite{fryntov}, the identity \eqref{stokes-ident} was applied to the functions $s = |\varphi|/\varphi$ and $s = |p\varphi|/\varphi$ to obtain the identities \eqref{length-1} and \eqref{length-2} respectively.  Here, we will adopt the slightly different choice
$$ s \coloneqq \frac{|\varphi|}{\varphi} \eta(|p|)$$
where $\eta: \R \to [0,1]$ is a fixed smooth cutoff to $[1/2,2]$ that equals $1$ on a neighborhood of $1$.
It is easy to see that $s$ is bounded by $O(1)$, has quasismooth derivative, and coincides with the outward unit normal $|\varphi|/\varphi$ on $\partial E_1$ away from critical points, so that \eqref{stokes-ident-2} applies.
From \eqref{wirt-5}, \eqref{wirt-6} and the product rule one has
$$
\partial s = -\frac{|\varphi|}{\varphi} \frac{\varphi'}{2\varphi} \eta(|p|) + \frac{|\varphi|}{\varphi} \eta'(|p|) \frac{\overline{p} p'}{2|p|}.
$$
The second term on the right-hand side is $O(|p'|)$.  Thus by \eqref{X1-def} and \eqref{stokes-ident-2} we have
$$ \ell(\partial E_1 \cap \Omega) \leq - \Re \int_{E_1} \frac{|\varphi|}{\varphi} \frac{\varphi'}{\varphi} \eta(|p|)\ dA + O( X_1 + X_5 )$$
(cf. \eqref{length-1}).
By \eqref{log-phi-expand} we have
$$ \frac{\varphi'}{\varphi} = \psi - \frac{p'}{p} = \psi + O(|p'|)$$
on the support of $\eta(|p|)$, thus by \eqref{X1-def} again we have
$$ \ell(\partial E_1 \cap \Omega) \leq - \Re \int_{E_2 \cap \Omega} \frac{|\varphi|}{\varphi} \psi \eta(|p|) 1_{E_1} dA + O( X_1 + X_5 )$$
(cf., \eqref{length-2}).
  Let $\sigma: \R \to [0,1]$ be a smooth cutoff to $[1/2,\infty)$ that equals $1$ on $[1,\infty)$ with derivative bounds $\sigma' = O(1)$.
By \eqref{X2-def}, we have
$$ \int_{E_2 \cap \Omega} \frac{|\varphi|}{\varphi} \psi \eta(|p|) 1_{E_1} \left(1-\sigma\left(\frac{\psi}{\lambda}\right)\right) \ dA \ll X_2$$
and thus
$$ \ell(\partial E_1 \cap \Omega) \leq - \Re \int_{E_2 \cap \Omega} \frac{|\varphi|}{\varphi} \psi \eta(|p|) 1_{E_1} \sigma\left(\frac{\psi}{\lambda}\right)\ dA + O( X_1 + X_2 + X_5 ).$$
Now we perform some manipulations inspired by \cite[\S 9]{fryntov}.  We factor
$$ \frac{|\varphi|}{\varphi} \psi = \frac{|p'|}{p'} \frac{p}{|p|} \frac{\psi}{|\psi|} |\psi| = e^{-i (\arg p' - \arg p - \arg \psi)} |\psi|$$
noting that the cutoffs force $p,p'$ to be non-zero; the argument is only defined up to multiples of $2\pi$, but this will not affect the calculations.  Thus we have
$$ \ell(\partial E_1 \cap \Omega) \leq - \int_{E_2 \cap \Omega} \cos(\arg p' - \arg p - \arg \psi) \eta(|p|) 1_{E_1} \sigma\left(\frac{\psi}{\lambda}\right) |\psi|\ dA + O( X_1 + X_2 + X_5 ).$$
Using the pointwise bound
$$ -\cos(\theta) 1_F \leq \cos_-(\theta)$$
for any indicator function $F$ and real number $\theta$, where $\cos_-(x) \coloneqq \max(-\cos(x),0)$, we conclude that
$$ \ell(\partial E_1 \cap \Omega) \leq \int_{E_2 \cap \Omega} \cos_-(\arg p' - \arg p - \arg \psi) \eta(|p|) \sigma\left(\frac{\psi}{\lambda}\right) |\psi|\ dA + O( X_1 + X_2 + X_5 ).$$
From \eqref{Psi-def} we have
$$ \int_{E_2 \cap \Omega} \frac{1}{\pi} \eta(|p|) \sigma\left(\frac{\psi}{\lambda}\right) |\psi|\ dA \leq  \Psi(E_2 \cap \Omega),$$
and so it suffices to show that
\begin{equation}\label{cos}
 \int_{E_2 \cap \Omega} \left(\cos_-(\arg p' - \arg p - \arg \psi)-\frac{1}{\pi}\right) \eta(|p|) \sigma\left(\frac{\psi}{\lambda}\right) |\psi|\ dA \ll X_1 + X_3 + X_4 + X_5.
\end{equation}
The function $x \mapsto \cos_-(x) - 1/\pi$ is $2\pi$-periodic with mean zero (cf. Figure \ref{fig:cos}).  Thus it has an antiderivative $F: \R \to \R$ that is also $2\pi$-periodic.  Observe from \eqref{wirt-9} and the chain rule that
\begin{align*}
    \partial F(\arg p' - &\arg p - \arg \psi) = \frac{1}{2i} F'(\arg p' - \arg p - \arg \psi) \left(\frac{p''}{p'} - \frac{p'}{p} - \frac{\psi'}{\psi} \right) \\
    &= \frac{1}{2i} \left(\cos_-(\arg p' - \arg p - \arg \psi)-\frac{1}{\pi}\right) \left(\psi + O(|p'|) + O\left(\frac{|\psi'|}{|\psi|}\right)\right)
\end{align*}
on the support of $\eta(|p|)$.  We therefore can write the integrand
$$ \left(\cos_-(\arg p' - \arg p - \arg \psi)-\frac{1}{\pi}\right) \eta(|p|) \sigma\left(\frac{\psi}{\lambda}\right) |\psi|$$
as
$$ 2i \frac{\eta(|p|) \sigma\left(\frac{\psi}{\lambda}\right) |\psi|}{\psi} \partial F(\arg p' - \arg p - \arg \psi)
+ O( |p'| ) + O\left(\frac{|\psi'|}{|\psi|}\right)$$
and thus by \eqref{X1-def}, \eqref{X3-def} one can majorize the left-hand side of \eqref{cos} by
$$ \ll \left| \int_{E_2 \cap \Omega} \frac{ \eta(|p|) \sigma\left(\frac{\psi}{\lambda}\right) |\psi|}{\psi}\partial F(\arg p' - \arg p - \arg \psi)\ dA \right| + X_1 + X_3.$$
The function $\frac{ \eta(|p|) \sigma\left(\frac{\psi}{\lambda}\right) |\psi|}{\psi}$ is bounded with quasismooth derivative and thus suitable for integration by parts.  Performing this integration, and noting that the only non-vanishing boundary contribution comes from $\partial \Omega$, and is $O(X_5)$ by \eqref{X5-def} and the boundedness of the both factors in the integration by parts formula, we can bound this by
$$ \ll \left| \int_{E_2 \cap \Omega} \left(\partial \left(\frac{\sigma\left(\frac{\psi}{\lambda}\right) |\psi|}{\psi} \eta(|p|)\right)\right) F(\arg p' - \arg p - \arg \psi)\ dA \right| + X_1 + X_3 + X_5.$$
The integrand is supported on the region where $|\psi| \geq \lambda/2$, and we can bound
$$ \partial\left( \frac{ \eta(|p|) \sigma\left(\frac{\psi}{\lambda}\right) |\psi|}{\psi}\right) \ll \frac{|\psi'|}{|\psi|} + \frac{|\lambda'|}{|\lambda|} + |p'|$$
(recalling that the cutoff $\eta(|p|)$ localizes to the region $|p| \asymp 1$).
Since $F$ is bounded, we can use \eqref{X1-def}, \eqref{X3-def}, \eqref{X4-def} to bound this expression by $O(X_1 + X_3 + X_4 + X_5)$ as desired.

\section{An initial bound on lemniscate length}\label{main-i-sec}

We now have all the tools to quickly prove Theorem \ref{main-thm}(i).  Let $p$ be a normalized maximizer of degree $n$.  We abbreviate $E_R(p)$ as $E_R$. From \eqref{polya-r} (or \eqref{area-formula-alt}) we have
\begin{equation}\label{e2-upper}
    \r[E_2] \leq \r[E_4] \leq 4^{1/n} = 1 + O\left( \frac{1}{n} \right).
\end{equation}
(The bounds on $E_4$ are not needed for this section, but will be useful in the next one.)

We apply Theorem \ref{stokes} with the constant function choice $\lambda = \sqrt{n}$ and $\Omega = \C$.  From \eqref{X4-def}, \eqref{X5-def} we have
$$ X_4 = X_5 = 0.$$
From \eqref{X2-def}, \eqref{e2-upper}, \eqref{E-area} we have
$$ X_2 \leq \sqrt{n} \pi \r[E_2]^2 \ll \sqrt{n}.$$
From \eqref{X1-def}, \eqref{e2-upper}, \eqref{area-form}, \eqref{E-area} and Cauchy--Schwarz we have
$$ X_1 \ll \sqrt{2\pi n \times 2^2} (\pi \r[E_2]^2)^{1/2} \ll \sqrt{n}.$$
From \eqref{X3-def}, \eqref{multip}, \eqref{log-psi-expand}, \eqref{e2-upper} we have
$$ X_3 \leq \frac{1}{\sqrt{n}} \int_{E_2} \sum_\zeta \frac{1}{|z-\zeta|} + \sum_{p''(\xi)=0} \frac{1}{|z-\xi|}\ dA(\xi) \ll n^{1/2} \r[E_2] \ll \sqrt{n}.$$
Theorem \ref{stokes} then gives
\begin{equation}\label{psi2}
\ell(\partial E_1) \leq \Psi(E_2) + O(\sqrt{n}).
\end{equation}
On the other hand, from \eqref{Psi-def}, \eqref{triangle}, \eqref{multip}, \eqref{e2-upper} we have
\begin{equation}\label{2n}
\Psi(E_2)= \frac{1}{\pi} \int_{E_2} |\psi|\ d A \leq
 \frac{1}{\pi} \sum_\zeta \int_{E_2} \frac{1}{|z-\zeta|}\ d A(z) \leq
 2 \r[E_2] (n-1) = 2n + O(1).
\end{equation}
Combining \eqref{psi2} and \eqref{2n}, we obtain Theorem \ref{main-thm}(i).

\section{An improved bound}\label{main-ii-sec}

In this section we refine the analysis of the preceding section to obtain Theorem \ref{main-thm}(ii).  Again, let $p$ be a normalized maximizing polynomial of degree $n$, and abbreviate $E_R(p)$ as $E_R$.  We may take $n$ to be sufficiently large, since the claim follows from (say) part (i) otherwise.

As $p$ is a normalized maximizer, we have
\begin{equation}\label{lupin}
 \ell(\partial E_1) \geq \ell(\partial E_1(p_0)) = 2n + 4 \log 2 + o(1)
\end{equation}
thanks to \eqref{ep0}.  Comparing this with \eqref{psi2}, \eqref{2n}, we conclude that
\begin{equation}\label{e2}
 \Psi(E_2) = 2n + O(1) - K
\end{equation}
for some $K$ with
\begin{equation}\label{ksqr}
1 \leq K \ll \sqrt{n}.
\end{equation}

The quantity $K$ measures the gain in the main term of Theorem \ref{stokes}; the strategy will be to show that the error terms $X_1,\dots,X_5$ in that theorem are of lower order than $K$.  To do this, we use several of the tools from previous sections to control some other quantities in terms of $K$.

\begin{lemma}[$K$ controls the geometry of $p$]\label{geomcontrol}\
    \begin{itemize}
        \item[(i)]  ($E_4$ is nearly round) We have
\begin{equation}\label{incl}
 D(0, 1-o(1)) \subset E_4 \subset D(0,1+o(1))
\end{equation}
with the perimeter bound
\begin{equation}\label{perimeter}
    \ell(\partial E_4) \ll K^{1/2}.
\end{equation}
In particular, from the Gauss--Lucas theorem we have
\begin{equation}\label{crit-bound}
    |\zeta| \leq 1 + o(1)
\end{equation}
for all critical points $\zeta$.
\item[(ii)]  (Dispersion bound)  We have (cf. Heuristic \ref{main-heuristic})
\begin{equation}\label{first}
\Disp \{\zeta\} = \|p\|_1 = \sum_\zeta |\zeta| \ll K.
\end{equation}
In particular, from \eqref{markov} one has
\begin{equation}\label{first-markov}
    \# \{ \zeta: \zeta \notin D(0,r) \} \ll \frac{K}{r}
\end{equation}
for all $r>0$.
\item[(iii)] (Exponential decay of $p'$ away from $\partial E_4$)  For any $z_0 \in E_4$, we have
\begin{equation}\label{exp-decay}
    p'(z_0) \ll \frac{e^{O(K)-c n \dist(z_0, \partial E_4)}}{\dist(z_0, \partial E_4)}
\end{equation}
for some absolute constant $c>0$.
\end{itemize}
\end{lemma}

\begin{proof}
By comparing \eqref{2n} with \eqref{e2} we have
$$ \r[E_2] \geq 1 - O\left(\frac{K}{n}\right).$$
Inserting this into the area theorem \eqref{area-formula-2}, we obtain
$$ \sum_{k=1}^\infty \frac{k |a_k|^2}{2^{2k/n}} \ll \frac{K}{n}.
$$
In particular, by Cauchy--Schwarz and \eqref{ksqr} we have
$$ \sum_{k=1}^\infty \frac{|a_k|}{4^{k/n}} \ll \frac{K^{1/2} \log^{1/2} n}{n^{1/2}} \ll \frac{\log^{1/2} n}{n^{1/4}} = o(1).$$
Applying \eqref{inclusions} we obtain \eqref{incl}.  Similarly, we have
$$ \sum_{k=1}^\infty \frac{k^2 |a_k|^2}{2^{4k/n}}  \ll n  \sum_{k=1}^\infty \frac{k |a_k|^2}{2^{2k/n}} \ll K,$$
and the claim \eqref{perimeter} then follows from \eqref{arc-bound}.

From Corollary \ref{defect-psi-cor} and \eqref{incl}, \eqref{crit-bound} one has
$$\Psi(E_2) \leq 2(n-1) \r[E_2] - c \Disp \{\zeta\}$$
for some absolute constant $c>0$.  From this and \eqref{disp-zeta} we conclude (ii).

Finally, for (iii) it suffices by the maximum principle (and \eqref{incl}) to verify the claim for $z_0 \in E_4 \backslash D(0,1/2)$. We apply from Lemma \ref{improv} with $s = \dist(z_0, \partial E_4)$ (the hypothesis \eqref{zr} being satisfied thanks to \eqref{ksqr}).  Since $D(z_0,s) \subset E_4$, we have $\sup_{z \in D(z_0,s)} |p(z)| \leq 4$, and hence Lemma \ref{improv} and \eqref{first} gives the desired bound (iv).
\end{proof}

Let $0 < \rho < 1/2$ be a small radius to be chosen later.
We will apply Theorem \ref{stokes} with $\Omega=\C \backslash \overline{D(0,\rho)}$ and the function
$$ \lambda(z) \coloneqq \frac{\mu n}{|z|}$$
where $0 < \mu < 1$ is a small parameter to be chosen later.  By \eqref{X5-def}, we have
\begin{equation}\label{x5}
    X_5 = 2\pi \rho.
\end{equation}
 By \eqref{X4-def} and direct calculation we conclude that
\begin{equation}\label{x4}
 X_4 \ll 1.
\end{equation}

Now we control the other error terms $X_1,X_2,X_3$.  Some of the bounds established here will be stated in more generality than needed for the current application, as they will also be useful in the next section.

\begin{lemma}[Control on $X_1$]  We have
    $$ X_1 \ll K^{7/8}.$$
\end{lemma}

The exponent $7/8$ can be improved, but the key point here is that the bound is of lower order than $O(K)$.

\begin{proof}
By \eqref{X1-def}, it suffices to show that
$$ \int_{E_2} |p'(z_0)|\ dA(z_0) \ll K^{7/8}.$$

From \eqref{perimeter}, the $K^{5/4}/n$-neighborhood of $\partial E_4$ has area $O(K^{7/4}/n)$.  From \eqref{area-form} and Cauchy--Schwarz, we conclude that the contribution to $X_1$ of $z_0$ with $\dist(z_0, \partial E_4) \leq K^{5/4}/n$ is acceptable.

For $m \geq 0$, the annular region $2^{m-1} K^{5/4}/n \leq \dist(z_0, \partial E_4) \leq 2^m K^{5/4}/n$ similarly has area $O(2^m K^{7/4}/n)$.  From \eqref{exp-decay} we have the pointwise bound
$$ |p'(z_0)| \ll \frac{e^{O(K)-c 2^{m-1} K^{5/4}}}{2^m K^{5/4}/n} $$
in this region.
The total contribution of all these regions to $X_1$ is then
$$ \ll \sum_{m=1}^\infty \frac{2^m K^{7/4}}{n} \frac{e^{O(K)-c 2^{m-1} K^{5/4}}}{2^m K^{5/4}/n} \ll K^{1/2} e^{O(K)-c K^{5/4}} \ll 1$$
which is also acceptable.
\end{proof}

\begin{lemma}[Control on $X_2$]\label{x2-lem}
If $\mu>0$ is sufficiently small, then
$$ X_2 \ll \mu \|p\|_1.$$
\end{lemma}

\begin{proof}  Let $C_0$ be a large constant (independent of $\mu$), and set $r_k \coloneqq 2^k C_0 \|p\|_1/n$ for $k=0,1,2,\dots$.  By \eqref{X2-def} and dyadic decomposition we have
\begin{equation}\label{x2-b}
     X_2 \ll \int_{D(0,r_0)} \frac{\mu n}{|z|}\ dA(z)+ \sum_{k \geq 1: r_k \ll 1} \frac{\mu n}{r_k} \left|\Ann(0,r_k,r_{k+1}) \cap \left\{|\psi| \leq \frac{\mu n}{r_k}\right\}\right|.
\end{equation}
The first term integrates to $O(\mu n r_0) = O(\mu C_0 \|p\|_1)$.  Now let $z$ be an element of the annulus
$\Ann(0,r_k,r_{k+1})$ with $|\psi(z)| \leq \frac{\mu n}{r_k}$, where $k \geq 1$ and $r_k \ll 1$. We use \eqref{psi-expand} to split $\psi = \psi_1 + \psi_2$, where
$$ \psi_1 \coloneqq \sum_{\zeta \in D(0, r_{k-1})} \frac{1}{z-\zeta}$$
and
$$ \psi_2 \coloneqq \sum_{\zeta \notin D(0, r_{k-1})} \frac{1}{z-\zeta}.$$
As $z$ lies in $\Ann(0,r_k,r_{k+1})$, we see from elementary geometry that the terms $\frac{1}{z-\zeta}$ contributing to $\psi_1$ all lie in a sector of aperture $2 \sin^{-1} \frac{r_{k-1}}{r_k} = \pi/3 < \pi$ and have magnitude $\asymp 1/r_k$, and there are $\gg n$ such terms by \eqref{markov-many}, \eqref{first} if $C$ is large enough.  From Lemma \ref{lower-triangle} we conclude that
$$ |\psi_1(z)| \gg \frac{n}{r_k}.$$
In particular, for $\mu$ small enough, the condition $|\psi(z)| \leq \mu n / r_k$ will force $|\psi_2(z)| \gg n/r_k$ from the triangle inequality.  If we let
$$ N_k \coloneqq | \{ \xi: \xi \notin D(0,r_{k-1}) \} |$$
be the number of critical points that contribute to $\psi_2$, then
from \eqref{distrib} we have
$$ \left|\Ann(0,r_k,r_{k+1}) \cap \left\{|\psi| \leq \frac{\mu n}{r_k}\right\}\right| \ll \frac{r_k^2}{n^2} N_k^2$$
and hence by \eqref{x2-b}
$$ X_2 \ll \mu C_0 \|p\|_1 + \mu \sum_{k=1}^\infty r_k N_k \frac{N_k}{n}.$$
By \eqref{first} we have
\begin{equation}\label{nksum}
    \sum_{k=1}^\infty r_k N_k \ll \|p\|_1
\end{equation}
and we trivially have $\frac{N_k}{n} \leq 1$.  The claim follows.
\end{proof}

\begin{lemma}[Control on $X_3$]\label{x3-lem}
    For any $r>0$, we have
$$ \int_{z \in D(0,r): |\psi(z)| \geq \mu n / |z|} \frac{|\psi'|}{|\psi|}\ dA
 \ll \frac{nr+\|p\|_1 \log n}{\mu n} .$$
In particular, taking $r=2$ (say), we have
$$ X_3 \ll \frac{n+\|p\|_1 \log n}{\mu n}.$$
\end{lemma}

\begin{proof}
We first deal with a technical contribution near the critical points $\zeta$.  In a disk $D(\zeta,rn^{-3})$, we can use \eqref{multip} and \eqref{log-psi-expand} to conclude
$$ \int_{D(\zeta,rn^{-3})} \frac{|\psi'|}{|\psi|}\ dA \ll rn^{-3} \times n.$$
Thus, the contribution of $\bigcup_\zeta D(\zeta,rn^{-3})$ to $X_3$ is $O(n \times rn^{-3} \times n)$, which is acceptable.  The remaining contribution can now be bounded by
$$ \int_{D(0,r) \backslash \bigcup_\zeta D(\zeta,rn^{-3})} \frac{|\psi'|}{\mu n/|z|}\ dA(z).$$
Using \eqref{psi-expand} and the triangle inequality, we can bound this by
$$
\ll \frac{1}{\mu n} \sum_\zeta \int_{D(0,r) \backslash D(\zeta,rn^{-3})} \frac{|z|}{|z-\zeta|^2}\ dA(z).$$
Direct computation (splitting into the regions $|z| \geq 2|\zeta|$ and $|z| < 2 |\zeta|$) shows that the integral is bounded by $O( r + |\zeta| \log n)$.  The claim now follows from \eqref{semi-norm}.
\end{proof}

Combining these three lemmas and \eqref{x4}, \eqref{e2} with Theorem \ref{stokes}, we conclude that
$$ \ell(\partial E_1 \backslash D(0,\rho)) \leq 2n - K - \Psi(D(0,\rho)) + O\left( K^{7/8} + \mu \|p\|_1 + \frac{n+\|p\|_1 \log n}{\mu n} + 1\right).$$
Selecting $\mu \coloneqq K^{-1/2}$ and using \eqref{ksqr}, \eqref{first}, we simplify to
\begin{equation}\label{lio}
     \ell(\partial E_1 \backslash D(0,\rho)) \leq 2n - K - \Psi(D(0,\rho)) + O( K^{7/8})
\end{equation}
and thus
\begin{equation}\label{lio-2}
 \ell(\partial E_1 \backslash D(0,\rho)) \leq 2n - \Psi(D(0,\rho)) + O(1).
\end{equation}
Sending $\rho \to 0$, we obtain Theorem \ref{main-thm}(ii).

\section{An even sharper bound}\label{main-iii-sec}

Now we prove Theorem \ref{main-thm}(iii).  Let $\eps > 0$ be a small parameter. Assume that $n$ is sufficiently large depending on $\eps$, and that $p$ is a normalized maximizing polynomial of degree $n$.  We abbreviate $E_R(p)$ as $E_R$. It will suffice to show that
\begin{equation}\label{lupin-targ}
 \ell(\partial E_1) = 2n + 4 \log 2 + O(\eps^{1/2}).
\end{equation}

We can control the total size of $\|p\|$ defined in \eqref{size-def}, in the spirit of Heuristic \ref{main-heuristic}:

\begin{proposition}[Total size bound]\label{lemni}  We have
\begin{equation}\label{quam-0}
    \|p\|_1 \ll 1
\end{equation}
and
\begin{equation}\label{quam-1}
    \|p\|_0 \ll 1
\end{equation}
and hence by \eqref{size-def}
\begin{equation}\label{quam}
 \|p\| \ll 1.
\end{equation}
In particular, from \eqref{quam-0}, \eqref{markov} one has
\begin{equation}\label{quam-markov}
    \# \{ \zeta: \zeta \notin D(0,r) \} \ll \frac{1}{r}
\end{equation}
for all $r>0$.
\end{proposition}

\begin{proof}
 As before, we have \eqref{lupin}.  Comparing this against the limiting case $\rho \to 0$ of \eqref{lio}, we conclude that the quantity $K$ from the previous section obeys the bound $K=O(1)$.  Thus, by \eqref{first}, one has \eqref{quam-0}.

Let $r \coloneqq \frac{C}{n}$, where $C$ is a large constant.  To show \eqref{quam-1} or \eqref{quam}, it suffices by Lemma \ref{origin-repulsion}, \eqref{quam-0} to show that $\partial E_1$ intersects $\overline{D(0,r)}$.  Suppose for contradiction that the lemniscate avoids $\overline{D(0,r)}$.  From \eqref{lio-2} we then have
$$  \ell(\partial E_1) \leq 2n - \Psi(D(0,r)) + O(1).$$
But from Lemma \ref{ankh} and \eqref{quam-0} one has
$$\Psi(D(0,r)) \geq \Psi(\Ann(0,r/2,r)) \geq (n-1) r - O(1),$$
and thus
$$ |\ell(\partial E_1)| \leq 2n - (n-1) r + O( 1 + r),$$
which contradicts \eqref{lupin} for $C$ large enough.  The claim follows.
\end{proof}

From Proposition \ref{pform} we have
\begin{equation}\label{pform-alt}
 p(z) = -1 + \frac{zp'(z)}{n} \left( 1 + O_r\left(\frac{1}{n \delta(z)} \right) \right)
\end{equation}
whenever $|z| \geq r > 0$ and $\delta(z) \gg 1/n$.
As an initial application of this formula, we obtain

\begin{lemma}\label{initial-apps} Let
\begin{equation}\label{sigma-def}
    \sigma \coloneqq \frac{C_0 \log n}{n},
\end{equation}
for a large constant $C_0>0$.  Then
\begin{equation}\label{e1in}
  E_4 \subset D(0, 1 + \sigma).
\end{equation}
In particular, by the Gauss--Lucas theorem (or the fact that $\partial E_1$ contains the critical points) we have
\begin{equation}\label{crit-bound-2}
    |\zeta| \leq 1 + \sigma
\end{equation}
for all critical points $\zeta$.  Also, we have
\begin{equation}\label{pp-bound}
    |p'(z)| \ll n
\end{equation}
for all $z \in E_4$.
\end{lemma}

\begin{proof}  From \eqref{incl} we already have
$$  E_4 \subset D(0, 1 + \eps).$$
Suppose that $z \in E_4 \backslash D(0,1)$ and $\delta(z) \geq C/n$ for a large constant $C$.  Then $|p(z)| \leq 4$ and $|z| \geq 1$, and by \eqref{pform-alt} we conclude that $|p'(z)| \ll n$ (if $C$ is large enough).  On the other hand, from \eqref{p'-bound-lower} we have $|p'(z)| \gg n |z|^{n-1} (C/n)^{O(1)}$.  Comparing the bounds we conclude that
$z \in D\left(0, 1 + O_C\left(\frac{\log n}{n} \right)\right)$.  We conclude from \eqref{quam-markov}, \eqref{delta-def} that $E_4$ is contained in the union of $D\left(0, 1 + O_C\left(\frac{\log n}{n} \right)\right)$ and $O(1)$ disks of radius $C/n$.  Since $E_4$ is connected, the claim \eqref{e1in} follows from the triangle inequality (if $C_0$ is large enough).

To prove \eqref{pp-bound}, we have already treated the cases where $z \in E_4 \backslash D(0,1)$ and $\delta(z) \geq C/n$, and from Proposition \ref{pvar} we also handle the case when $z \in D(0,1)$.  The remaining case is when $z \in E_4 \backslash D(0,1)$ and $\delta(z) < C/n$.  As discussed before, the restriction $\delta(z) < C/n$ localizes $z$ to $O(1)$ disks of radius $C/n$.  By the connectedness of $E_4$ (and \eqref{incl}) we thus see that $z$ lies within $O(C/n)$ of a point $z' \in E_4$ with $\delta(z') = C/n$, so that $p'(z') \ll n$ by the previous analysis.  From \eqref{pp-factor} we have
$$ \left|\frac{p'(z)}{p'(z')}\right| = \prod_\zeta \frac{|z-\zeta|}{|z'-\zeta|}.$$
Since $\delta(z') = C/n$ and $|z-z'| \ll C/n$, we see that each factor here is $O(1)$, and for $\zeta \in D(0,1/2)$ the factor is $1 + O(|\zeta| n / C) = \exp(O(|\zeta|))$.  Multiplying using \eqref{quam-0}, we conclude \eqref{pp-bound} in this case also.
\end{proof}

Let $0 < \rho < \eps$ be a small radius. From the above lemma we may split
\begin{equation}\label{split}
 \ell(\partial E_1 \backslash D(0,\rho)) = \ell(\partial E_1 \cap \Ann(0,\rho,\eps)) + \ell(\partial E_1 \cap \Ann(0,\eps,1-\sigma)) + \ell(\partial E_1 \cap \Ann(0,1-\sigma,1+\sigma)).
\end{equation}
We now estimate the three components separately.  We begin with the innermost component:

\begin{proposition}[Inner bound]\label{inside}  We have
    $$ \ell(\partial E_1 \cap \Ann(0,\rho,\eps)) \leq 2\eps n - \Psi(D(0,\rho)) + O(\eps^{1/2}).$$
\end{proposition}

\begin{proof}  We apply Theorem \ref{stokes} with $\Omega = D(0,\eps)$ and $\lambda(z) = \mu n / |z|$ for a small parameter $\mu > 0$ to be chosen later (mollifying $\lambda$ at the origin as before).  We conclude (using \eqref{incl}) that
$$ \ell(\partial E_1 \cap \Ann(0,\rho,\eps)) \leq \Psi(D(0,\eps)) - \Psi(D(0,\rho)) + O( X_1 + X_2 + X_3 + X_4 + X_5)$$
where $X_1,\dots,X_5$ are defined as in \eqref{X1-def}--\eqref{X5-def}.  Direct calculation gives
$$ X_4, X_5 \ll \eps $$
while from Proposition \ref{pvar} we have
$$ X_1 \ll n\eps^{n-1} e^{O(1/\eps)} \times \pi \eps^2.$$
From Lemma \ref{x2-lem} and \eqref{quam-0} we have
$$ X_2 \ll \mu.$$
while from Lemma \ref{x3-lem} and \eqref{quam-0} one has
$$ X_3 \ll \frac{n\eps + \log n}{\mu n}.$$
Finally, we apply \eqref{triangle}, \eqref{multip} to obtain
$$ \Psi(D(0,\eps)) \leq 2 (n-1) \eps.$$
Setting $\mu = \eps^{1/2}$ and combining all the above bounds, we obtain the claim after a brief calculation.
\end{proof}

To treat the contributions near critical points, we first need

\begin{lemma}[Local upper bound on lemniscate length]\label{sting}  If $z_0 \neq 0$ and $0 < r < |z_0|/2$, then
$$ \ell(\partial E_1 \cap D(z_0, r)) \ll \frac{nr^2}{|z_0|} + r + \frac{r\|p\|_1}{|z_0|} + \frac{\|p\|_1 \log n}{n}.$$
In particular, if $|z_0| \gg \|p\|_1$ and $r \gg \frac{\log^{1/2} n}{n} |z_0|$, then we have the simplified bound
$$ \ell(\partial E_1 \cap D(z_0, r)) \ll \frac{nr^2}{|z_0|}.$$
\end{lemma}

\begin{proof}  From \eqref{incl} we may assume that $z_0 = O(1)$, otherwise the set $\partial E_1 \cap D(z_0,r)$ is empty.

    We apply Theorem \ref{stokes} with $\Omega = D(z_0,r)$ and $\lambda = n/|z_0|$.  We conclude that
$$ \ell(\partial E_1 \cap D(z_0,r)) \ll \Psi(D(z_0,r)) + O( X_1 + X_2 + X_3 + X_4 + X_5)$$
where $X_1,\dots,X_5$ are defined as in \eqref{X1-def}--\eqref{X5-def}.  Direct calculation gives
$$ X_2 \ll \frac{n r^2}{|z_0|}; \quad X_4 = 0; \quad X_5 \ll r$$
From \eqref{X1-def} and \eqref{pp-bound} we have
$$ X_1 \ll n r^2 \ll \frac{nr^2}{|z_0|}.$$
For $X_3$, we first control a minor contribution when $z \in \bigcup_\zeta D(\zeta, n^{-2} r)$.  From \eqref{multip}, \eqref{log-psi-expand} this contribution is bounded by $O(n \times n \times n^{-2} r) = O(r)$.  The remaining contribution to $X_3$ is bounded by
$$ \frac{|z_0|}{n} \int_{D(z_0,r) \backslash \bigcup_\zeta D(\zeta,n^{-2} r)} |\psi'|\ dA(z).$$
By \eqref{psi-expand} and the triangle inequality, this is bounded by
$$ \frac{|z_0|}{n} \sum_\zeta \int_{D(z_0,r) \backslash D(\zeta,n^{-2} r)} \frac{1}{|z-\zeta|^2}\ dA(z)$$
Each integral is $O(\log n)$, but with the improved bound of $O( r^2/|z_0|^2 )$ if $\zeta \in D(0,|z_0|/4)$.  By \eqref{markov}, this bound holds for all but $O(\|p\|_1/|z_0|)$ critical points, hence
$$ X_3 \ll r + \frac{|z_0|}{n} n \frac{r^2}{|z_0|^2} + \frac{|z_0|}{n} \frac{\|p\|_1}{|z_0|} \log n.$$
Finally, from \eqref{Psi-def}, \eqref{triangle} we have
\begin{equation}\label{tri-use}
  \Psi(D(z_0,r)) \leq \sum_\zeta \frac{1}{\pi} \int_{D(z_0,r)} \frac{1}{|z-\zeta|}\ dA(z).
\end{equation}
The integral here is $O(r)$ for all $\zeta$ (by \eqref{riesz-bound}), but the bound can be improved to $r^2 / |z_0|$ if $\zeta \in D(0, |z_0|/4)$.  By \eqref{markov}, this condition $\zeta \in D(0,|z_0|/4)$ holds for  all but $O(\|p\|_1/|z_0|)$ critical points, thus
$$ \Psi(D(z_0,r)) \ll n \frac{r^2}{|z_0|} + \frac{\|p\|_1}{|z_0|} r.$$
Combining all these estimates, we obtain the claim.
\end{proof}

Now we handle the intermediate region:

\begin{proposition}[Bound in intermediate region]\label{annulus}  We have
    $$ \ell(\partial E_1 \cap \Ann(0,\eps,1-\sigma)) \leq 2n (1-\eps-\sigma) + O_\eps\left(\frac{\log n}{n} \right).$$
\end{proposition}

\begin{proof}
By \eqref{quam-markov}, there are $O(1/\eps)$ critical points $\zeta$ for which $D(\zeta, 2\frac{\log^{1/2} n}{n})$ intersects $\Ann(0,\eps,1-\sigma)$; by the triangle inequality, each of these critical points obeys $|\zeta| \asymp_\eps 1$.  By Lemma \ref{sting} and Proposition \ref{lemni}, the combined contribution of these disks $D(\zeta, 2\frac{\log^{1/2} n}{n})$ to the lemniscate length is $O_\eps(\log n/n)$.
It will thus suffice to establish the bound
$$ \ell(\partial E_1 \cap \Omega) \leq 2n (1-\eps-\sigma) + O_\eps\left(\frac{\log n}{n} \right)$$
where
$$ \Omega \coloneqq \left\{ z \in \Ann(0,\eps,1-\sigma): \delta(z) \geq \frac{\log^{1/2} n}{n} \right\}.$$
We now plan to apply the second arclength formula \eqref{arcl-2} from Lemma \ref{arclength}.
From Proposition \ref{pform}, $p(z) \neq -1$, and
$$ \arg \frac{p(z)}{zp'(z)} = \arg \frac{p(z)}{p(z)+1}
+ O_\eps\left(\frac{1}{n \delta(z)} \right) \pmod \pi.$$
Since $|p(z)|=1$ and $p(z) = -1 + O((1-\sigma)^n) = -1 + O(1/n)$, we have
$$ \arg(p(z)) = \pi + O\left(\frac{1}{n}\right) \pmod{2\pi}; \quad \arg(p(z)+1) = \pm \frac{\pi}{2} + O\left(\frac{1}{n}\right) \pmod{2\pi};$$
and thus (by \eqref{delta-upper})
$$ \arg \frac{p(z)}{zp'(z)} = \pm \frac{\pi}{2} + O_\eps\left(\frac{1}{n \delta(z)} \right) \pmod{2\pi}.$$
We can thus apply \eqref{arcl-2}, \eqref{cosec-asym} to conclude that
$$
 \ell(\partial E_1 \cap \Omega) = \int_\eps^{1-\sigma} \sum_{z \in \partial E_1 \cap \Omega \cap \partial D(0,r)} \left(1 + O_\eps\left( \frac{1}{n^2 \delta(z)^2} \right)\right)\ dr.
$$
By the discussion in Section \ref{conj-sec}, we have
$$ \sum_{z \in \partial E_1 \cap \Omega \cap \partial D(0,r)} 1 \leq 2n$$
which implies that
\begin{equation}\label{zero-use}
    \int_{\eps}^{1-\sigma}\sum_{z \in \partial E_1 \cap \Omega \cap \partial D(0,r)} 1 \ dr \leq \int_\eps^{1-\sigma} 2n\ dr = 2n(1-\eps-\sigma).
\end{equation}
It thus suffices to show that
\begin{equation}\label{ute}
     \int_\eps^{1-\sigma} \sum_{z \in \partial E_1 \cap \Omega \cap \partial D(0,r)} \frac{1}{\delta(z)^2}\ dr \ll_\eps n \log n.
\end{equation}
Bounding
$$ \frac{1}{\delta(z)^2} \ll_\eps 1 + \sum_{\zeta \notin D(0,\eps/2)} \frac{1}{|z - \zeta|^2}$$
and recalling from \eqref{quam-markov} that there are only $O(1/\eps)$ critical points outside of $D(0,\eps/2)$, it suffices to show that
$$ \int_\eps^{1-\sigma} \sum_{z \in \partial E_1 \cap \Omega \cap \partial D(0,r)} \frac{1}{|z - \zeta|^2}\ dr \ll_\eps n \log n$$
for any such critical point.  By dyadic decomposition, it suffices to show that
$$ \int_\eps^{1-\sigma} \# (\partial E_1 \cap \Omega \cap \partial D(0,r) \cap D(\zeta, r_1)) \ dr \ll_\eps n r_1^2$$
for all $\frac{\log^{1/2} n}{n} \leq r_1 \leq 1$.  Using \eqref{arcl-2} again, we can lower bound the left-hand side by
$$ \ll_\eps \ell( \partial E_1 \cap \Omega \cap D(\zeta,r_1)),$$
and the required bound follows from Lemma \ref{sting} (covering $D(\zeta,r_1)$ by smaller disks if $r_1$ is large).
\end{proof}

Now we handle the outer region.

\begin{proposition}[Outer bound]\label{outside}  We have
    $$ \ell(\partial E_1 \cap \Ann(0,1-\sigma,1+\sigma)) \leq 2n \sigma + 4 \log 2 + O_{C_0}\left( \frac{1}{\log n} \right).$$
\end{proposition}

\begin{proof}  A technical difficulty here is that the function $\delta$ could be very small on some portions of the annulus $\Ann(0,1-\sigma,1+\sigma)$, due to the presence of nearby critical points.  Because of this, we will need to remove an exceptional set when $\delta$ is too small, and need quite precise control on $p$ and $p'$ outside of this exceptional set to compute the lemniscate length in that region to the desired accuracy.  As such, the proof here will be among the most complicated in the paper.

We turn to the details.  From \eqref{pp-factor} we have
\begin{equation}\label{p1}
 p'(z) = n z^{n-1} g(1/z)
\end{equation}
where $g$ is the polynomial
$$ g(w) \coloneqq \prod_\zeta \left( 1 - w \zeta \right).$$
For $w \in D(0,3)$, we have $1 - w\zeta \leq \exp( O(|\zeta|) )$, hence by \eqref{quam} we have $g=O(1)$ on this disk; by the Cauchy inequality we thus have $g'=O(1)$ on $D(0,2)$.  Also, $g(0)=1$, and by \eqref{quam-markov} there are only $O(1)$ zeroes of $g$ in $D(0,2)$.  From Jensen's inequality we conclude that
$$ \int_0^{2\pi} \log |g(e^{-i\theta})|\ d\theta \geq 0.$$
We subdivide $[0,2\pi]$ into $\asymp \log^4 n$ intervals $I$ of length $\asymp \frac{1}{\log^4 n}$.  If we let $c_I$ be the mean value of $\log |g(e^{-i\theta})|$ on $I$, we conclude that
\begin{equation}\label{sumi}
 \sum_I c_I |I| \geq 0.
\end{equation}
Call an interval $I$ \emph{good} if $\delta(e^{i\theta}) \geq \frac{1}{\log^2 n}$ for at least one $\theta \in I$, and \emph{bad} otherwise.  If $I$ is good, then we have $|1- e^{i\theta} \zeta| \gg \frac{1}{\log^2 n}$ for all $\theta \in I$ and critical points $\zeta$, while from \eqref{quam} we have $|1 - e^{i\theta} \zeta| \asymp 1$ for all but $O(1)$ of the $\zeta$. From \eqref{quam} and the triangle inequality we have
\begin{equation}\label{gpp}
 \frac{g'(e^{i\theta})}{g(e^{i\theta})} = \sum_\zeta \frac{-\zeta}{1 - e^{i\theta} \zeta} \ll \log^2 n,
\end{equation}
and thus $\log |g(e^{-i\theta})|$ is $O(\log^2 n)$-Lipschitz on $I$, and thus we have
\begin{equation}\label{loggi}
    \log |g(e^{-i\theta})| = c_I + O\left(\frac{1}{\log^2 n}\right)
\end{equation}
For any $I$, direct calculation shows that
$$ \frac{1}{|I|} \int_I \log |1-w\zeta|\ d\theta \ll \log\log n$$
for any critical point, with an improved bound of $O(|\zeta|)$ if $\zeta \in D(0,1/2)$.  By \eqref{quam-markov} we thus have the bound
\begin{equation}\label{cilog}
     c_I \ll \log\log n
\end{equation}
for all intervals $I$.
Let $\Omega$ consist of those $z = re^{i\theta}$ with $1-\sigma \leq r \leq 1+\sigma$ and $\theta$ lying in a good interval $I$.
From \eqref{quam-markov} we see that the arcs $\{e^{i\theta}: \theta \in I\}$ corresponding to bad $I$ are contained in $O(1)$ disks of radius $O(\frac{1}{\log^2 n})$.  Hence the portion of $\partial E_1 \cap \Ann(0,1-\sigma,1+\sigma)$ lying outside of $\Omega$ can be covered by $O(\frac{1}{\sigma \log^2 n})$ disks of radius $\sigma$, whose center lies in $\Ann(0,1/2,2)$ (say), and hence by Lemma \ref{sting} and Proposition \ref{lemni} its contribution to the lemniscate length is
$$\ll  \frac{1}{\sigma \log^2 n} \times n \sigma^2 \ll_{C_0} \frac{1}{\log n}.$$
Thus it will suffice to show that
$$ \ell(\partial E_1 \cap \Omega) \leq 2n \sigma + 4 \log 2 + O\left(\frac{1}{\log n} \right).$$
From the first arclength formula \eqref{arcl-1} from Lemma \ref{arclength}, we can bound
$$ \ell(\partial E_1 \cap \Omega) \leq \int_{-\pi}^{\pi} \sum_{z \in \overline{\Omega}: p(z) = e^{i\alpha}} \frac{1}{|p'(z)|} \ d\alpha.$$
For $z = re^{i\theta}$ contributing to the above integral, we have $1-\sigma \leq r \leq 1+\sigma$, $\theta$ in a good interval $I$, and by \eqref{p1}, \eqref{loggi} we have
\begin{equation}\label{ppz}
 |p'(z)| = n |z|^{n-1} e^{c_I} \left( 1 + O\left( \frac{1}{\log^2 n} \right) \right).
\end{equation}
From \eqref{pform-alt} and \eqref{p1}, \eqref{loggi} we have
$$ |1+e^{i\alpha}| = |1+p(z)| = |z|^n e^{c_I} \left( 1 + O\left( \frac{1}{\log^2 n} \right) \right)$$
and hence
\begin{equation}\label{zi}
 |z| = e^{-c_I/n} |1 + e^{i\alpha}|^{1/n} \left( 1 + O\left( \frac{1}{n \log^2 n} \right) \right).
\end{equation}
Substituting this into \eqref{ppz} we have
$$ |p'(z)| = n |1 + e^{i\alpha}|^{\frac{n-1}{n}} e^{c_I/n} \left( 1 + O\left( \frac{1}{\log^2 n} \right) \right)$$
which by \eqref{cilog} simplifies to
$$ |p'(z)| = n |1 + e^{i\alpha}|^{\frac{n-1}{n}} \left( 1 + O\left( \frac{1}{\log^2 n} \right) \right).$$
Hence
$$ \ell(\partial E_1 \cap \Omega) \leq \left(1 + O\left( \frac{1}{\log^2 n} \right)\right) \int_{-\pi}^{\pi} |1 + e^{i\alpha}|^{\frac{n-1}{n}} \frac{\# \{ z \in \overline{\Omega}: p(z) = e^{i\alpha} \}}{n} \ d\alpha.$$
It will thus suffice to show that
$$ \int_{-\pi}^{\pi} |1 + e^{i\alpha}|^{\frac{n-1}{n}} \frac{\# \{ z \in \overline{\Omega}: p(z) = e^{i\alpha} \}}{n} \ d\alpha \leq 2n\sigma + 4 \log 2 + O\left( \frac{1}{\log n} \right).$$
We split the left-hand side as
$$ \sum_{I \text{ good}} \int_{-\pi}^\pi |1 + e^{i\alpha}|^{\frac{n-1}{n}} \frac{m(I,\alpha)}{n} \ d\alpha$$
where
$$ m(I,\alpha) \coloneqq \# \{ re^{i\theta}: 1-\sigma \leq r \leq 1+\sigma, \theta \in I, p(re^{i\theta}) = e^{i\alpha} \}.$$
By \eqref{sumi}, it will suffice to show that
$$ \int_{-\pi}^\pi |1 + e^{i\alpha}|^{\frac{n-1}{n}} \frac{m(I,\alpha)}{n} \ d\alpha
\leq \left(2n\sigma - c_I + 4 \log 2 + O\left( \frac{1}{\log n} \right)\right) |I|$$
for each good interval $I$.
(Note from \eqref{cilog} that the right-hand side is non-negative for the bad intervals $I$, so one can add them to that side when summing over $I$.)

Fix a good interval $I$. If there is an $re^{i\theta}$ contributing to the quantity $m(I,\alpha)$, then by \eqref{zi} one has
\begin{align}
 |1 + e^{i\alpha}|^{1/n} &\geq e^{c_I/n} (1-\sigma) \left( 1 - O\left( \frac{1}{n \log^2 n} \right) \right) \nonumber\\
 &= 1 - \sigma + \frac{c_I}{n}- O\left( \frac{1}{n \log^2 n} \right).\label{wibble}
\end{align}
If we can show that
\begin{equation}\label{count}
     m(I,\alpha) \leq n|I| \left( 1 + O\left( \frac{1}{\log^2 n} \right) \right)
\end{equation}
for such $\alpha$, then the claim will follow from Lemma \ref{out} and \eqref{ep0}.  Thus it remains to show \eqref{count}.

Let $\theta_0$ be the midpoint of $I$.  Integrating \eqref{gpp}, we have
$$ g\left(\frac{1}{re^{i\theta}}\right) = g(e^{-i\theta_0}) \left( 1 + O\left(\frac{1}{\log^2 n} \right)\right)$$
whenever $1-2\sigma \leq r \leq 1+2\sigma$ and $\theta = \theta_0 + O(\frac{1}{\log^4 n})$.
From \eqref{pform-alt} and \eqref{p1}, we conclude that
\begin{equation}\label{pra}
 p(re^{i\theta}) = -1 + r^n e^{in\theta} g(e^{-i\theta_0}) \left( 1 + O\left(\frac{1}{\log^2 n} \right)\right)
\end{equation}
for such $r, \theta$.  From \eqref{loggi}, \eqref{cilog} one has $\log |g(e^{-i\theta_0})| \ll \log\log n$, thus
\begin{equation}\label{glog}
 \log^{-O(1)} n \ll |g(e^{-i\theta_0})| \ll \log^{O(1)} n.
\end{equation}

One can enclose $I$ in an interval $[\theta_-, \theta_+]$ of length $|I| + O\left(\frac{1}{n} \right) = |I| \left( 1 + O\left(\frac{1}{\log^2 n} \right) \right)$ such that $e^{in\theta_\pm} g(e^{-i\theta_0})$ is a positive real.  For $z$ in the sector
$$\Gamma \coloneqq \{re^{i\theta}: 1-2\sigma \leq r \leq 1+2\sigma, \theta \in [\theta_-, \theta_+] \}$$
we see from \eqref{pra}, \eqref{loggi} that we can split
$$ p(z) - e^{i\alpha} = f(z) + h(z)$$
where
$$ f(z) \coloneqq z^n g(e^{-i\theta_0}) - (1+e^{i\alpha})$$
and
$$ |h(z)| \ll \frac{|z|^n e^{c_I}}{\log^2 n}.$$
Direct counting shows that the number of zeroes to $f(z)$ (i.e., the $n^{\mathrm{th}}$ roots of $(1+e^{i\alpha}) / g(e^{-i\theta_0})$) in $\Gamma$ is at most $n|I| \left( 1 + O\left( \frac{1}{\log^2 n} \right) \right)$.  By Rouche's theorem, it will now suffice to show the pointwise bound $|h(z)| < |f(z)|$ for $z$ on the boundary of $\Gamma$.  When $|z| = 1 + 2\sigma$, we see from \eqref{loggi} and the definition of $\sigma$ that
\begin{equation}\label{fz}
    |f(z)| \gg |z|^n e^{c_I}
\end{equation}
which gives the claim in this case.  if $z$ lies on one of the rays $\theta = \theta_\pm$, then $z^n g(e^{-i\theta_0})$ is a positive real and so we again obtain \eqref{fz}, again giving the claim in this case.  Finally, when $|z| = 1 - 2\sigma$, then from \eqref{wibble}, \eqref{glog}, and \eqref{cilog} one has
$$ |f(z)| \asymp |1+e^{i\alpha}| \gg (1-\sigma)^n e^{c_I}$$
while $h(z) \ll (1-2\sigma)^n e^{c_I}$, giving the claim from the definition of $\sigma$.
\end{proof}

Inserting Propositions \ref{inside}, \ref{annulus}, \ref{outside} into \eqref{split}, we obtain
\begin{equation}\label{dip}
 \ell(\partial E_1 \backslash D(0,\rho)) \leq 2n + 4 \log 2 - \Psi(D(0,\rho)) + O(\eps^{1/2}) + O_\eps\left(\frac{\log n}{n} \right) + O_{C_0}\left( \frac{1}{\log n} \right).
\end{equation}
Sending $\rho \to 0$, and noting that $\eps$ can be made arbitrarily small, we obtain Theorem \ref{main-thm}(iii).

\section{The final result}\label{main-iv-sec}

We are now ready to establish the last and strongest component of Theorem \ref{main-thm}, namely part (iv).  Assume that $n$ is sufficiently large, and that $p$ is a normalized maximizing polynomial of degree $n$.  As before we abbreviate $E_r(p)$ as $E_r$. We will show that $\|p\|=0$, and hence $p = p_0$.

We may assume for sake of contradiction that
\begin{equation}\label{p-contra}
    \|p\| > 0.
\end{equation}
The first step is to obtain the following improvement of \eqref{quam-0}, again in the spirit of Heuristic \ref{main-heuristic}.

\begin{proposition}\label{ets}  We have
    $$ \|p\|_1 = o(1).$$
    In particular, we have $\zeta=o(1)$ for all critical points $\zeta$.
\end{proposition}

\begin{proof} Let $\eps > 0$ be a small constant. It will suffice to show that
\begin{equation}\label{sumzeta}
 \|p\|_1 \ll \eps^{1/2}
\end{equation}
if $n$ is sufficiently large depending on $\eps$.

From \eqref{lupin} and the hypothesis that $p$ is a maximizing polynomial, one has
\begin{equation}\label{eo}
     \ell(\partial E_1) \geq \ell(\partial E_1(p_0)) = 2n + 4 \log 2 + o(1).
\end{equation}
Comparing the above bound with the analysis in the previous section (particularly Propositions \ref{inside}, \ref{annulus}, \ref{out}), we see that for any small absolute constant $\eps>0$ one must have
\begin{align}
    \ell(\partial E_1 \cap D(0,\eps)) &= 2n\eps  + O(\eps^{1/2}) \label{est-1} \\
     \ell(\partial E_1 \cap \Ann(0,\eps,1-\sigma)) &= 2n (1-\eps-\sigma) +  O(\eps^{1/2}) \label{est-2} \\
      \ell(\partial E_1 \cap \Ann(0,1-\sigma,1+\sigma)) &= 2n \sigma + 4 \log 2 + O(\eps^{1/2}) \label{est-3}
\end{align}
if $n$ is sufficiently large depending on $\eps$.  We can use this information to get further control on the critical points $\zeta$, improving upon \eqref{quam}.

Firstly, by comparing \eqref{est-1} with the proof of Proposition \ref{inside} (focusing in particular on \eqref{tri-use}), we see that
$$ \Psi(D(0,\eps)) = (n-1) 2 \eps - O(\eps^{1/2}),$$
and hence by Corollary \ref{defect-psi-cor} we have
\begin{equation}\label{zet-1}
 \Disp \{ \zeta: \zeta \in D(0,10\eps)\} \ll \eps^{1/2}.
\end{equation}

Next, comparing \eqref{est-2} with the proof of Proposition \ref{annulus} (focusing in particular on \eqref{zero-use}), we have
$$ \int_{\eps}^{1-\sigma} \sum_{z \in \partial E_1 \cap \Omega \cap \partial D(0,r)} 1 \ dr = 2n(1-\eps-\sigma) - O(\eps^{1/2})$$
or equivalently
$$ \int_{\eps}^{1-\sigma} \left(2n - \sum_{z \in \partial E_1 \cap \Omega \cap \partial D(0,r)} 1 \right) \ dr \ll \eps^{1/2}.$$
We will shortly show the lower bound
\begin{equation}\label{lower}
 2n - \sum_{z \in \partial E_1 \cap \Omega \cap \partial D(0,r)} 1  \geq \# \{ \zeta: |\zeta| \geq r + \eps \}
\end{equation}
for any $\eps \leq r \leq 1/2$.  Assuming this, we conclude that
$$ \sum_\zeta \int_{\eps}^{1/2} 1_{|\zeta| \geq r + \eps}\ dr \ll \eps^{1/2}.$$
Performing the inner integral for $\zeta \notin D(0,10\eps)$ (and using \eqref{crit-bound-2}) we have
\begin{equation}\label{zet-2}
     \sum_{\zeta\notin D(0,10\eps)} |\zeta| \ll \eps^{1/2}.
\end{equation}
Combining \eqref{zet-1} and \eqref{zet-2}  using Lemma \ref{disp-split}, we obtain \eqref{sumzeta} as desired.

It remains to show \eqref{lower} for a given $\eps \leq r \leq 1/2$. Observe that if $r+\eps/4 \leq r' \leq r+\eps/2$ is such that $\delta(z) \leq \delta_0$ for some $\delta_0 > 0$ and $z \in \partial D(0,r')$, then (by \eqref{quam-markov}) $r'$ will lie in the union of $O_\eps(1)$ intervals of length $O_\eps(\delta_0)$.  Setting $\delta_0$ to be sufficiently small depending on $\eps$, the measure of this union is strictly less than the measure of the interval $[r+\eps/4, r+\eps/2]$.  We conclude that we can find a radius $r \leq r' \leq r+\eps/2$ such that $\delta(z) \gg_\eps 1$ for all $z \in \partial D(0,r')$.

From Section \ref{conj-sec}, the condition $|p(z)| = 1$ is equivalent to $z^n p(z) \tilde p(r^2/z) - z^n = 0$.  It will suffice to show that the polynomial $z^n p(z) \tilde p(r^2/z) - z^n$ has at most $2n - k$ zeroes in $D(0,r')$, where $k$ is the number of critical points $\zeta$ with $|\zeta| \geq r + \eps$.

For $z \in \partial D(0,r')$, we can use \eqref{pform-alt} (to control $p(z)$) and \eqref{pin-2} (to control $\tilde p(r^2/z)$) to write
$$ z^n p(z) \tilde p(r^2/z) - z^n$$
as
$$ z^n \left( -1 + \frac{zp'(z)}{n} \left( 1 + O_\eps\left(\frac{1}{n} \right) \right) \right)
\left( -1 + O_\eps( (r^2/r')^n ) \right) - z^n.$$
From \eqref{pvar}, \eqref{quam} and the hypothesis $\delta(z) \gg_\eps 1$ we have
$$ |p'(z)| \asymp_\eps n |z|^{n-1} = n (r')^n$$
and hence the $O_\eps( (r^2/r')^n )$ term can be absorbed into the $O_\eps(1/n)$ term since $r' \geq r + \eps/4$.  Thus we can simplify the previous expression to
$$ - z^n \frac{zp'(z)}{n} \left( 1 + O_\eps\left(\frac{1}{n} \right) \right).$$
By Rouche's theorem, the number of zeroes of this expression inside $D(0,r')$ is equal to the number of zeroes of
$$ - z^n \frac{zp'(z)}{n},$$
inside the same disk, which by \eqref{pp-factor} is equal to $2n$ minus the number of critical points outside of $D(0,r')$.  Since this latter number is at least $k$, we obtain the claim.
\end{proof}

Now we can improve Proposition \ref{lemni}:

\begin{proposition}\label{pots}  We have
    $$ \|p\| = o(1).$$
\end{proposition}

\begin{proof}
 Let $C>0$ be a large constant, and let $\eps>0$ be small.   By Proposition \ref{ets} and Lemma \ref{origin-repulsion}, it suffices to show that lemniscate $\partial E_1$ enters $D(0,C \eps^{1/2}/n)$ for $n$ sufficiently large depending on $C,\eps$.

Suppose this is not the case.  Setting $\rho = C\eps^{1/2}/n$ in \eqref{dip}, we conclude that
$$ \ell(\partial E_1) \leq 2n + 4 \log 2 - \Psi(D(0,C\eps^{1/2}/n)) + O(\eps^{1/2}) + O_\eps\left(\frac{\log n}{n} \right) + O_{C_0}\left( \frac{1}{\log n} \right),$$
which on comparison with \eqref{eo} gives
$$ \Psi(D(0,C\eps^{1/2}/n)) \ll \eps^{1/2}$$
for $n$ large enough.  But from Lemma \ref{ankh} and Proposition \ref{ets} we have
$$ \Psi(D(0,C\eps^{1/2}/n)) \geq (n-1) \frac{C\eps^{1/2}}{2n} - o(1),$$
giving the required contradiction if $C$ is large enough.
\end{proof}

Let $C_0$ be a large constant to be chosen later. Inspired by the arguments in \cite[\S 6]{fryntov}, and in the spirit of Heuristic \ref{main-heuristic}, we will now seek to establish the inequality
\begin{equation}\label{a-ineq}
    \ell( \partial E_1(p) ) \leq \ell( \partial E_1(p_0) ) - c \|p\| + O\left(\frac{\|p\|}{C_0}\right)
\end{equation}
for some absolute constant $c>0$ (independent of $C_0$), for $n$ large enough; from this, \eqref{eo}, and Proposition \ref{pots} we conclude (for $C_0$ large enough) that $\|p\|=0$, contradicting \eqref{p-contra}.

It remains to establish \eqref{a-ineq}. Set
$$ r_- \coloneqq \frac{\|p\|}{C_0^2}; \quad r_+ \coloneqq C_0^2 \|p\|.$$
Note from Proposition \ref{pots}, \eqref{p-contra} that $0 < r_- < r_+ = o(1)$.
Similarly to \eqref{split}, we split
\begin{equation}\label{split-2}
 \ell(\partial E_1) = \ell(\partial E_1 \cap D(0,r_-)) + \ell(\partial E_1 \cap \Ann(0,r_-,r_+)) + \ell(\partial E_1 \cap \Ann(0,r_+,1+\sigma)),
\end{equation}
and estimate the three components separately.

We begin with an analogue of Proposition \ref{inside}, which achieves additional gains consistent with Heuristic \ref{main-heuristic}.

\begin{proposition}[Inner bound]\label{inside-2}  We have
\begin{equation}\label{in2-first}
     \ell(\partial E_1 \cap D(0,r_-)) \leq 2n r_- - c \Disp \{ \zeta : \zeta \in D(0,10 r_-) \} + O\left(\frac{\|p\|}{C_0}\right)
\end{equation}
for an absolute constant $c>0$ (independent of $C_0$).  If furthermore
\begin{equation}\label{pp0}
 \|p\| \geq C \|p\|_1,
\end{equation}
for a large absolute constant $C$ (independent of $C_0, c$), we can obtain the refinement
\begin{equation}\label{in2-second}
     \ell(\partial E_1 \cap D(0,r_-)) \leq 2n r_- - c \Disp \{ \zeta : \zeta \in D(0,10 r_-) \} - c \|p\| + O\left(\frac{\|p\|}{C_0}\right)
\end{equation}
\end{proposition}

\begin{proof} As in the proof of Proposition \ref{inside}, we apply Theorem \ref{stokes} with $\Omega = D(0,r_-)$ and\footnote{Technically, one cannot directly apply Theorem \ref{stokes} here due to the singularity of $\lambda$ at the origin, but this is easily dealt with by replacing $\Omega$ with $\Ann(0,\rho,r_-)$ and then taking the limit $\rho \to 0$.} $\lambda(z) = \mu n / |z|$ for $\mu \coloneqq 1/C_0$, to conclude
$$ \ell(\partial E_1 \cap D(0,r_-)) \leq \Psi(E_2 \cap D(0,r_-)) + O( X_1 + X_2 + X_3 + X_4 + X_5)$$
where $X_1,\dots,X_5$ are defined as in \eqref{X1-def}--\eqref{X5-def}.  Direct calculation gives
$$ X_4, X_5 \ll r_- \ll \frac{\|p\|_1}{C_0} $$
while from Proposition \ref{pvar}(ii) and Proposition \ref{ets} we have
$$ X_1 \ll e^{O(n)} n^{-n} \|p\|_1^{n-1} \times \pi r_-^2 \ll r_- \ll \frac{\|p\|_1}{C_0}.$$
From Lemma \ref{x2-lem} we have
$$ X_2 \ll \mu \|p\|_1 = \frac{\|p\|}{C_0}.$$
From Lemma \ref{x3-lem} we have
$$ X_3 \ll \frac{nr_- + \|p_0\| \log n}{\mu n} \ll \frac{\|p\|}{C_0}.$$
We conclude that
$$ \ell(\partial E_1 \cap D(0,r_-)) \leq \Psi(D(0,r_-)) + O\left(\frac{\|p\|}{C_0}\right).$$
From Corollary \ref{defect-psi-cor} we have
$$ \Psi(D(0,r_-)) \leq 2 (n-1) r_- - c \Disp \{ \zeta : \zeta \in D(0,10 r_-) \}$$
for an absolute constant $c>0$.
The claim \eqref{in2-first} follows.

Suppose now that \eqref{pp0} holds.  From Lemma \ref{origin-repulsion} we conclude (for $C$ large enough) that the lemniscate $\partial E_1$ does not enter $D(0,\|p\|/4n)$.  We now repeat the above arguments, but with $\Omega$ set equal to $\Ann(0,\|p\|/4n,r_-)$ rather than $D(0,r_-)$, to obtain
$$ \ell(\partial E_1 \cap D(0,r_-)) \leq \Psi(E_2 \cap \Ann(0,\|p\|/4n,r_-)) + O( X_1 + X_2 + X_3 + X_4 + X_5)$$
From Lemma \ref{ankh} we have
$$ \Psi(D(0, \|p\|/4n)) \geq c\| p \|$$
for some absolute constant $c>0$, and hence
$$\Psi(E_2 \cap \Ann(0,\|p\|/4n,r_-)) \leq \Psi(E_2 \cap D(0,r_-)) - c \|p\|.$$
Repeating the previous arguments, we obtain the claim.
\end{proof}

Now we obtain an analogue of Proposition \ref{annulus} (with a version of the gain implicit in the proof of Proposition \ref{ets}, and suggested by Heuristic \ref{main-heuristic}).

\begin{proposition}[Bound in intermediate region]\label{annulus-2}  We have
    $$ \ell(\partial E_1 \cap \Ann(0,r_-,r_+)) \leq 2n (r_+-r_-) - c \sum_{\zeta \notin D(0,10r_-)} |\zeta| + O_{C_0}\left(\frac{\log n}{n} \|p\| \right)$$
for an absolute constant $c>0$ (independent of $C_0$).
\end{proposition}

\begin{proof}
By \eqref{markov}, there are $O_{C_0}(1)$ critical points $\zeta$ for which $D(\zeta, 2\frac{\log^{1/2} n}{n} r_+)$ intersects $\Ann(0,r_-, r_+)$; by the triangle inequality, each of these critical points obeys $|\zeta| \asymp_{C_0} \|p\| \asymp_{C_0} r_- \asymp_{C_0} r_+$.  By Lemma \ref{sting}, the combined contribution of these disks $D(\zeta, 2\frac{\log^{1/2} n}{n} r_+)$ to the lemniscate length is $O_{C_0}(\log n \|p\|/n)$.
It will thus suffice to establish the bound
$$ \ell(\partial E_1 \cap \Omega) \leq 2n (r_+-r_-) + O_{C_0}\left(\frac{\log n}{n} \|p\| \right)$$
where
$$ \Omega \coloneqq \left\{ z \in \Ann(0,r_-,r_+): \delta(z) \geq \frac{\log^{1/2} n}{n} \right\}.$$
We again aim to apply the second estimate \eqref{arcl-2} from Lemma \ref{arclength}.  Suppose that $z \in \partial E_1 \cap \Omega$.
From Proposition \ref{pform}, $p(z) \neq -1$, and
$$ \arg \frac{p(z)}{zp'(z)} = \arg \frac{p(z)}{p(z)+1}
+ O_{C_0}\left(\frac{1}{n \delta(z)} \right) \pmod \pi.$$
Since $|p(z)|=1$ and $p(z) = -1 + O_{C_0}(r_+^n) = -1 + O(1/n)$, we have
$$ \arg(p(z)) = \pi + O\left(\frac{1}{n}\right) \pmod{2\pi}; \quad \arg(p(z)+1) = \pm \frac{\pi}{2} + O\left(\frac{1}{n}\right) \pmod{2\pi};$$
and thus by \eqref{delta-upper}
$$ \arg \frac{p(z)}{zp'(z)} = \pm \frac{\pi}{2} + O_{C_0}\left(\frac{1}{n \delta(z)} \right) \pmod{2\pi}.$$
We can thus apply \eqref{arcl-2}, \eqref{cosec-asym} to conclude that
$$
 \ell(\partial E_1 \cap \Omega) = \int_{r_-}^{r_+} \sum_{z \in \partial E_1 \cap \Omega \cap \partial D(0,r)} 1 + O_{C_0}\left( \frac{1}{n^2 \delta(z)^2} \right)\ dr.
$$

Similarly to \eqref{ute}, we claim that
\begin{equation}\label{ute-2}
     \int_{r_-}^{r_+} \sum_{z \in \partial E_1 \cap \Omega \cap \partial D(0,r)} \frac{1}{\delta(z)^2}\ dr \ll_{C_0} n \|p\| \log n
\end{equation}
Indeed, bounding
$$ \frac{1}{\delta(z)^2} \ll_{C_0} 1 + \sum_{\zeta \notin D(0,r_-/2)} \frac{\|p\|^2}{|z - \zeta|^2}$$
for any $z$ contributing to the above expression,
and recalling from \eqref{markov} that there are only $O_{C_0}(1)$ critical points outside of $D(0,r_-/2)$, it suffices to show that
$$ \int_{r_-}^{r_+} \sum_{z \in \partial E_1 \cap \Omega \cap \partial D(0,r)} \frac{1}{|z - \zeta|^2}\ dr \ll_{C_0} \frac{n \log n}{\|p\|}$$
for any such critical point.  By dyadic decomposition, it suffices to show that
$$ \int_{r_-}^{r_+} \# (\partial E_1 \cap \Omega \cap \partial D(0,r) \cap D(\zeta, r_1)) \ dr \ll_{C_0} \frac{n r_1^2}{\|p\|}$$
for all $\frac{\log^{1/2} n}{n} \|p\| \ll_{C_0} r_1 \ll_{C_0} \|p\|$.  Using \eqref{arcl-2} again, we can lower bound the left-hand side by
$$ \ll_{C_0} \ell( \partial E_1 \cap \Omega \cap D(\zeta,r_1)),$$
and the required bound follows from Lemma \ref{sting} (covering $D(\zeta,r_1)$ by smaller disks if $r_1$ is large).

In view of \eqref{ute-2}, it will suffice to show that
$$ \int_{r_-}^{r_+} \left(2n - \sum_{z \in \partial E_1 \cap \Omega \cap \partial D(0,r)} 1\right)\ dr \gg \sum_{\zeta \notin D(0,10r_-)} |\zeta|.$$
Analogously to the proof of Proposition \ref{ets}, it will suffice to establish the lower bound
$$ 2n - \sum_{z \in \partial E_1 \cap \Omega \cap \partial D(0,r)} 1 \geq \# \{ \zeta: |\zeta| \geq r + r_- \}$$
for all $r_- \leq r \leq r_+$ (noting that all the critical points $\zeta$ have magnitude at most $\|p\|_1 \leq r_+$).

Continuing the proof of Proposition \ref{ets}, for any $r_- \leq r \leq r_+$, we can find a radius $r \leq r' \leq r+r_-/2$ such that $\delta(z) \gg_{C_0} 1$ for all $z \in \partial D(0,r')$, and it will suffice to show that the polynomial $z^n p(z) \tilde p(r^2/z) - z^n$ has at most $2n - k$ zeroes in $D(0,r')$, where $k$ is the number of critical points $\zeta$ with $|\zeta| \geq r + r_-$.  But this follows from Rouche's theorem by repeating the rest of the proof of Proposition \ref{ets} (with implied constants now depending on $C_0$ instead of $\eps$).
\end{proof}

Now we handle the outer region.

\begin{proposition}[Outer bound]\label{outside-again}  We have
    $$ \ell(\partial E_1 \cap \Ann(0,r_+,1+\sigma)) \leq \ell(\partial E_1(p_0)) - 2nr_+ + O\left( \frac{\|p\|}{C_0} \right).$$
\end{proposition}

\begin{proof}  This is similar to Proposition \ref{outside}, but the proof is simpler because the annulus $\Ann(0,r_+,1+\sigma)$ is known to be free of critical points.
Indeed, all the critical points $\zeta$ obey $|\zeta| \leq \|p\|_1 \leq r_+/C_0^2$, so in particular $\delta(z) \asymp 1$ for all $z \notin D(0,r_+)$.  By the first formula \eqref{arcl-1} of Lemma \ref{arclength}, we have
$$ \ell(\partial E_1 \cap \Ann(0,r_+,1+\sigma)) \leq \int_{-\pi}^{\pi} \sum_{z \in \overline{\Ann(0,r_+,1+\sigma)}: p(z) = e^{i\alpha}} \frac{1}{|p'(z)|} \ d\alpha.$$
For $z$ in the above sum, we have from Proposition \ref{pform} that
$$ e^{i\alpha} = p(z) = -1 + \frac{zp'(z)}{n} \left( 1 + O\left(\frac{\|p\|^2}{n |z|^2} \right) \right)$$
and hence
$$ |p'(z)| = \frac{n |1 + e^{i\alpha}|}{|z|} \left( 1 + O\left(\frac{\|p\|^2}{n |z|^2} \right) \right).$$
Also, from Proposition \ref{pvar} one has
$$ |p'(z)| = n |z|^{n-1} \left( 1 + O\left(\frac{\|p\|^2}{|z|^2} \right) \right)$$
and thus
$$ |1+e^{i\alpha}| = |z|^n \left( 1 + O\left(\frac{\|p\|^2}{|z|^2} \right) \right).$$
This implies that
$$ |z| \asymp |1+e^{i\alpha}|^{1/n}$$
and in fact we have the more precise
$$ |z| = |1 + e^{i\alpha}|^{1/n} \left( 1 + O\left(\frac{\|p\|^2}{n |1+e^{i\alpha}|^{2/n}} \right) \right)$$
while the condition $|z| \geq r_+$ implies
$$ |1 + e^{i\alpha}|^{1/n} \geq r_+ \left( 1 - O\left(\frac{1}{n C_0^4} \right) \right).$$
By the fundamental theorem of algebra, there are at most $n$ choices of $z$ associated to each $\alpha$. We thus have
$$ \ell(\partial E_1 \cap \Ann(0,r_+,1+\sigma)) \leq \int_{I_r} |1 + e^{i\alpha}|^{-\frac{n-1}{n}} \left( 1 + O\left(\frac{\|p\|^2}{n |1+e^{i\alpha}|^{2/n}} \right) \right)\ d\alpha$$
where
$$r \coloneqq r_+ \left( 1 - O\left(\frac{1}{n C_0^4} \right) \right) = r_0 - O\left( \frac{\|p\|}{n C_0^2} \right)$$
and $I_r$ was defined in Lemma \ref{out}.  From that lemma we have
$$ \int_{I_r} |1 + e^{i\alpha}|^{-\frac{n-1}{n}}\ d\alpha = \ell(\partial E_1(p_0)) - 2n r + O\left( \frac{\|p\|}{C_0^2} \right)$$
and so to obtain an acceptable error term, it will suffice to show that
$$ \int_{I_r} |1 + e^{i\alpha}|^{-\frac{n+1}{n}}\ d\alpha \ll \frac{n}{C_0 \|p\|}.$$
By symmetry it suffices to show that
$$ \int_{I_r \cap [0,\pi]} |1 + e^{i\alpha}|^{-\frac{n+1}{n}}\ d\alpha \ll \frac{1}{C_0 \|p\|}.$$
Making the change of variables $2 \sin(\alpha/2) = u^n$ as in the proof of Lemma \ref{out}, the left-hand side becomes
$$ \int_r^{2^{1/n}} \frac{n du}{u^2 \sqrt{1 - u^{2n}/4}};$$
this integrates to $O(n/r) = O(n/C_0^2 \|p\|)$, giving the claim.
\end{proof}

If \eqref{pp0} holds, then from Propositions \ref{inside-2}, \ref{annulus-2}, \ref{outside-again} and \eqref{split-2} (and discarding some terms with a favorable sign) we have
$$ \ell(\partial E_1(p)) \leq \ell(\partial E_1(p_0)) - c \|p\| + O\left(\frac{\|p\|}{C_0}\right)$$
giving the claim \eqref{a-ineq}.  If instead \eqref{pp0} fails, we instead have
$$ \ell(\partial E_1(p)) \leq \ell(\partial E_1(p_0)) - c \Disp \{ \zeta : \zeta \in D(0,10 r_-) \} - c \sum_{\zeta \notin D(0,10r_-)} |\zeta| + O\left(\frac{\|p\|}{C_0}\right)$$
and hence by Lemma \ref{disp-split}
$$ \ell(\partial E_1(p)) \leq \ell(\partial E_1(p_0)) - c \|p\|_1 + O\left(\frac{\|p\|}{C_0}\right)$$
after adjusting $c$ as necessary.  The claim \eqref{a-ineq} now follows from the failure of \eqref{pp0}.  Thus \eqref{a-ineq} is proven in all cases, giving Theorem \ref{main-thm}(iv).

\bibliographystyle{amsplain}

\end{document}